\pdfoutput=1
\RequirePackage{ifpdf}
\ifpdf 
\documentclass[pdftex]{sigma}
\else
\documentclass{sigma}
\fi

\numberwithin{equation}{section}

\newtheorem{thm}{Theorem}[section]
\newtheorem{prp}[thm]{Proposition}
\newtheorem{lem}[thm]{Lemma}

{ \theoremstyle{definition}
\newtheorem{rem}[thm]{Remark}
\newtheorem{dfn}[thm]{Definition}
}

\begin{document}

\allowdisplaybreaks

\renewcommand{\thefootnote}{$\star$}

\newcommand{\arXivNumber}{1412.8001}

\renewcommand{\PaperNumber}{100}

\FirstPageHeading

\ShortArticleName{Tableau Formulas for One-Row Macdonald Polynomials of Types $C_n$ and $D_n$}

\ArticleName{Tableau Formulas for One-Row Macdonald\\ Polynomials of Types $\boldsymbol{C_n}$ and $\boldsymbol{D_n}$\footnote{This paper is a~contribution to the Special Issue
on Orthogonal Polynomials, Special Functions and Applications.
The full collection is available at \href{http://www.emis.de/journals/SIGMA/OPSFA2015.html}{http://www.emis.de/journals/SIGMA/OPSFA2015.html}}}

\Author{Boris {FEIGIN}~$^{\dag^1}$, Ayumu {HOSHINO}~$^{\dag^2}$, Masatoshi {NOUMI}~$^{\dag^3}$, Jun {SHIBAHARA}~$^{\dag^4}$\\ and Jun'ichi {SHIRAISHI}~$^{\dag^5}$}

\AuthorNameForHeading{B.~Feigin, A.~Hoshino, M.~Noumi, J.~Shibahara and J.~Shiraishi}

\Address{$^{\dag^1}$~National Research University Higher School of Economics, International Laboratory\\
\hphantom{$^{\dag^1}$}~of Representation Theory and Mathematical Physics, Moscow, Russia}
\EmailDD{\href{mailto:borfeigin@gmail.com}{borfeigin@gmail.com}}

\Address{$^{\dag^2}$~Kagawa National College of Technology,
355 Chokushi-cho, Takamatsu,\\
\hphantom{$^{\dag^2}$}~Kagawa 761-8058, Japan}
\EmailDD{\href{mailto:hoshino@t.kagawa-nct.ac.jp}{hoshino@t.kagawa-nct.ac.jp}}

\Address{$^{\dag^3}$~Department of Mathematics, Kobe University, Rokko, Kobe 657-8501, Japan}
\EmailDD{\href{mailto:noumi@math.kobe-u.ac.jp}{noumi@math.kobe-u.ac.jp}}

\Address{$^{\dag^4}$~Hamamatsu University School of Medicine,
1-20-1 Handayama, Higashi-ku,\\
\hphantom{$^{\dag^4}$}~Hamamatsu city, Shizuoka 431-3192, Japan}
\EmailDD{\href{mailto:s070521math@yahoo.co.jp}{s070521math@yahoo.co.jp}}

\Address{$^{\dag^5}$~Graduate School of Mathematical Sciences, University of Tokyo, Komaba,\\
\hphantom{$^{\dag^5}$}~Tokyo 153-8914, Japan}
\EmailDD{\href{mailto:shiraish@ms.u-tokyo.ac.jp}{shiraish@ms.u-tokyo.ac.jp}}

\ArticleDates{Received May 19, 2015, in f\/inal form November 26, 2015; Published online December 05, 2015}

\Abstract{We present explicit formulas for the Macdonald polynomials of
types $C_n$ and $D_n$ in the one-row case.
In view of the combinatorial structure,
we call them ``tableau formulas''.
For the construction of the tableau formulas,
we apply some transformation formulas for the basic hypergeometric series
involving very well-poised balanced ${}_{12}W_{11}$ series.
We remark that the correlation functions of the deformed~$\mathcal{W}$ algebra
generators automatically give rise to the tableau formulas when we principally specialize the
coordinate variables.}

\Keywords{Macdonald polynomials; deformed $\mathcal{W}$ algebras}

\Classification{33D52; 33D80}

\renewcommand{\thefootnote}{\arabic{footnote}}
\setcounter{footnote}{0}

\section{Introduction}

I.G.~Macdonald introduced the symmetric polynomials $P_{\lambda}(x; q,t)$ as a
$(q,t)$-deformation of the Schur polynomials $s_{\lambda}(x)$.
Then he extended this construction to the cases of the symmetric Laurent polynomials
invariant under the actions of the Weyl groups of simple root systems~\cite{Mac}.
For type $A_n$,
he gave an explicit combinatorial formula for $P_{\lambda}(x; q,t)$, usually called the
``tableau formula''.
In~\cite{FHSSY} it was shown that the tableau formula for $P_{\lambda}(x; q,t)$
of type $A_n$ can be interpreted as certain specialization of
the correlation functions of the deformed $\mathcal{W}$
algebras of type~$A_n$.
One of our motivations is to explore a little further the correspondence between the
Macdonald polynomials and the deformed $\mathcal{W}$ algebras associated
with simple root systems.
More precisely, we calculate the correlation functions of $\mathcal{W}$ algebras of
types $C_n$ and $D_n$ with principal specializations in coordinate variables
(Def\/inition \ref{DefOfCorrFunc}). As a result, we obtain certain combinatorial expressions,
which we regard as tableau formulas.
Then, on the basis of Lassalle's explicit formulas~\cite{La} (see~\cite{HNS}  for a proof and their generalization),
we prove that they are actually the Macdonald polynomials of types~$C_n$ and~$D_n$
in the one-row case (Theorem~\ref{soukan1}).

We need to recall brief\/ly the Kashiwara--Nakashima tableaux.
In the study of quantum algebras~\cite{KN}, they gave combinatorial descriptions of the crystal bases for the integrable highest weight
representations by using the ``semi-standard tableaux'' of
respective types.
For the crystal bases of the symmetric tensor representations $V(r \Lambda_1)$ of
types $C_n$ and $D_n$, the semi-standard tableau is def\/ined to be a
one-row diagram $(r)$ of size $r$ f\/illed with entries in the ordered set
$I=\{1,2, \dots, n, {\overline{n}}, {\overline{n-1}}, \dots, {\overline{1}} \}$
being arranged in the weakly increasing manner.
The orderings
of $I$ are def\/ined by~(\ref{OrderOfC}) for $C_n$ and by (\ref{OrderOfD}) for $D_n$ respectively.
For a semi-standard tableau of shape $(r)$,
denote by $\theta_i$ the number of the letter $i \in I$ in the tableau. Then we have
$\theta_1 + \theta_2 + \cdots + \theta_n + \theta_{\overline{n}} + \theta_{\overline{n-1}} + \cdots + \theta_{\overline{1}} = r$.
For type $D_n$, we have an aditional condition $\theta_n \theta_{\overline{n}} =  0$
due to the structure of the ordering.

We now present the main results of this paper (Theorems~\ref{main},~\ref{main_C}
and~\ref{main_C2}).
Let $P_{(r)}^{(C_n)}\!(x;q,t,T)$ and $P_{(r)}^{(D_n)}(x;q,t)$ be the Macdonald polynomials of
types $C_n$ and $D_n$ respectively
attached to a~single row~$(r)$. As
for the notation, see Appendix~\ref{appendixA.2}.

\begin{thm}
We have the following tableau formulas
in the one-row cases:
\begin{gather}
 P_{(r)}^{(C_n)}\big(x;q,t,t^2/q\big)=
\frac{(q;q)_r}{(t;q)_r}
    \displaystyle\sum_{\theta_{1}+\theta_{2}+\cdots+
    \theta_{\overline{1}}=r}
\prod_{k \in I}
\frac{(t;q)_{\theta_{k}}}{(q;q)_{\theta_{k}}}
\nonumber
\\
 \hphantom{P_{(r)}^{(C_n)}\big(x;q,t,t^2/q\big)=}{}
  \times \prod_{1\leq l\leq n}\frac{(t^{n-l+1}q^{\theta_{l}+
    \cdots+\theta_{\overline{l+1}}};q)_{\theta_{\overline{l}}}
    (t^{n-l+2}q^{\theta_{l+1}+\cdots+\theta_{\overline{l+1}}-1};q)_
    {\theta_{\overline{l}}}}
    {(t^{n-l+2}q^{\theta_{l}+\cdots+\theta_
    {\overline{l+1}}-1};q)_{\theta_{\overline{l}}}
    (t^{n-l+1}q^{\theta_{l+1}+\cdots+\theta_{\overline{l+1}}};q)_
    {\theta_{\overline{l}}}}\nonumber\\
 \hphantom{P_{(r)}^{(C_n)}\big(x;q,t,t^2/q\big)=}{}
  \times
    x_1^{\theta_{1}-\theta_{\overline{1}}} x_2^{\theta_{2}-\theta_{\overline{2}}}\cdots
    x_n^{\theta_{n}-\theta_{\overline{n}}}, \label{MacOfCt^2/q} \\
P_{(r)}^{(D_n)}(x;q,t)=
\frac{(q;q)_r}{(t;q)_r}\sum_{
\theta_1 + \theta_2 + \cdots +\theta_{\overline{1}} = r \atop
\theta_n\theta_{\overline{n}}=0}
\prod_{k \in I}
\frac{(t;q)_{\theta_{k}}}{(q;q)_{\theta_{k}}} \nonumber\\
\hphantom{P_{(r)}^{(D_n)}(x;q,t)=}{}
 \times \prod_{1 \leq l \leq n-1}\frac{(t^{n-l-1} q^{\theta_{l} + \theta_{l+1} + \cdots
 + \theta_{\overline{l+1}} +1 } ;q)_{\theta_{\overline{l}}}(t^{n-l} q^{\theta_{l+1} + \theta_{l+2} + \cdots
 + \theta_{\overline{l+1}}} ;q)_{\theta_{\overline{l}}}}
{(t^{n-l} q^{\theta_{l} + \theta_{l+1} + \cdots + \theta_{\overline{l+1}} } ;q)_{\theta_{\overline{l}}}(t^{n-l-1} q^{\theta_{l+1} + \theta_{l+2} + \cdots + \theta_{\overline{l+1}} + 1} ;q)_{\theta_{\overline{l}}}
}\nonumber\\
\hphantom{P_{(r)}^{(D_n)}(x;q,t)=}{}
\times
x_1^{\theta_1 - \theta_{\overline{1}}}x_2^{\theta_2 - \theta_{\overline{2}}}\cdots
x_n^{\theta_n - \theta_{\overline{n}}}.  \label{MacOfD}
\end{gather}
\end{thm}
Here and hereafter, we use the standard notation of $q$-shifted factorials
\begin{gather*}
 (z;q)_\infty =\prod_{k=0}^{\infty} (1-q^k z),
\qquad
(z;q)_k=\frac{(z;q)_{\infty}}{(q^kz;q)_{\infty}},\qquad k\in\mathbb{Z},\\
 (a_1, a_2, \dots, a_r;q)_k = (a_1;q)_k (a_2;q)_k \cdots (a_r;q)_k,
\qquad k\in\mathbb{Z}.
\end{gather*}

\begin{rem}
These tableau formulas for
 $P_{(r)}^{(C_n)}(x;q,t,t^2/q)$ and $P_{(r)}^{(D_n)}(x;q,t)$
are obtained by principally specializing the correlation functions of the deformed~$\mathcal{W}$
 algebras of types~$C_{n}$ and~$D_n$ respectively. See Theorem~\ref{soukan1}.
\end{rem}

We can extend the tableau
formula of type $C_n$ in~(\ref{MacOfCt^2/q})
as follows to general~$q$,~$t$ and~$T$.

\begin{thm}
Set $\theta := {\rm min}(\theta_{n}, \theta_{\overline{n}})$ for simplicity of display.
We have
\begin{gather}
P_{(r)}^{(C_n)}(x;q,t,T)=
\frac{(q;q)_r}{(t;q)_r}
    \displaystyle\sum_{\theta_{1}+\theta_{2}+\cdots+
    \theta_{\overline{1}}=r}
\prod_{k \in I {\setminus} \{ n, \overline{n} \} }
\frac{(t;q)_{\theta_{k}}}{(q;q)_{\theta_{k}}}
\frac{(t;q)_{|\theta_{n} - \theta_{\overline{n}}|}}{(q;q)_{|\theta_{n} - \theta_{\overline{n}}|}}
\nonumber
\\
\hphantom{P_{(r)}^{(C_n)}(x;q,t,T)=}{}
     \times \prod_{1\leq l\leq n-1}
\Biggl(
    {(t^{n-l-1}q^{\theta_{l}+
    \cdots+\theta_{n-1}+|\theta_{n} - \theta_{\overline{n}}|+\theta_{\overline{n-1}}+\cdots +
   \theta_{\overline{l+1}} +1};q)_{\theta_{\overline{l}}}
    \over
    (t^{n-l}q^{\theta_{l}+\cdots+\theta_{n-1}+|\theta_{n} - \theta_{\overline{n}}|+\theta_{\overline{n-1}}
+\cdots+
\theta_
    {\overline{l+1}}};q)_{\theta_{\overline{l}}}
  }\nonumber\\
  \hphantom{P_{(r)}^{(C_n)}(x;q,t,T)=}{}
 \times
    {    (t^{n-l}q^{\theta_{l+1}+\cdots\theta_{n-1}+|\theta_{n} - \theta_{\overline{n}}|
   +\theta_{\overline{n-1}}+\cdots+\theta_{\overline{l+1}}};q)_
    {\theta_{\overline{l}}}
    \over
    (t^{n-l-1}q^{\theta_{l+1}+\cdots+
\theta_{n-1}+|\theta_{n} - \theta_{\overline{n}}|+\theta_{\overline{n-1}}+\cdots
+\theta_{\overline{l+1}}+1};q)_
    {\theta_{\overline{l}}} } \Biggr)\nonumber\\
  \hphantom{P_{(r)}^{(C_n)}(x;q,t,T)=}{}
\times
{(T;q)_{\theta}(t^n q^{r-2\theta};q)_{2\theta} \over (q;q)_{\theta}(T t^{n-1} q^{r-\theta};q )_{\theta}
(t^{n-1} q^{r-2\theta +1};q)_{\theta}}
    x_1^{\theta_{1}-\theta_{\overline{1}}} x_2^{\theta_{2}-\theta_{\overline{2}}}\cdots
    x_n^{\theta_{n}-\theta_{\overline{n}}}.\label{C-ext}
\end{gather}
\end{thm}

\begin{rem}
At present,
we do not know any $\mathcal{W}$ algebra which explains  the
formula (\ref{C-ext}).
\end{rem}

There are several combinatorial expressions for the Macdonald polynomials
studied from some dif\/ferent points of view.
See \cite{Le,RY, DE} for example. It would be an intriguing problem to f\/ind possible connections between those formulas and ours
obtained in this paper.

This paper is organized as follows. In Sections~\ref{section2}, \ref{section3} and~\ref{section5} we construct the
transformation formulas for the basic hypergeometric series for proving
our tableau
formulas.
In Sections~\ref{section4} and~\ref{section6}, we prove the tableau formulas for the one-row Macdonald polynomials of
types $C_n$ and $D_n$ respectively.
In Section~\ref{section7},
we recall the deformed~$\mathcal{W}$ algebras of types~$C_n$
and~$D_n$, and then prove that the correlation functions with principal specialization
give us the  tableau formulas for the
Macdonald polynomials
of types~$C_n$ and~$D_n$ in the one-row case respectively.
In Appendix~\ref{appendixA},
we recall brief\/ly the Koornwinder polynomials, and the Macdonald polynomials
of types~$C_n$ and~$D_n$
as degenerations of the Koornwinder polynomials.

Throughout this paper,
we use the standard notation for the basic hypergeometric series
as
\begin{gather*}
{}_{r+1}\phi_r\left[ {a_1,a_2,\ldots,a_{r+1}\atop b_1,\dots,b_{r}};q,z\right]=
\sum_{n=0}^\infty {(a_1, a_2, \dots, a_{r+1};q)_n \over
(q, b_1, b_2, \dots, b_{r};q)_n} z^n,\\
{}_{r+1}W_r(a_1;a_4,a_5,\ldots,a_{r+1};q,z) =
{}_{r+1}\phi_r\left[ {a_1,q a_1^{1/2}, -q a_1^{1/2},a_4,\ldots,a_{r+1}\atop
a_1^{1/2},-a_1^{1/2},q a_1/a_4,\ldots,q a_1/a_{r+1}};q,z\right].
\end{gather*}
We call the ${}_{r+1}W_r$ series {\it very well-poised} basic hypergeometric series.
Moreover,
we call a~${}_{r+1}W_r$ series
{\it very well-poised balanced}
when it satisf\/ies the balancing condition
$(a_4a_5\cdots a_{r+1})z = \big({\pm} (a_1q)^{1 \over 2}\big)^{r-3}$.

\section{Transformation formula I}\label{section2}

In this section, we give a transformation formula of basic hypergeometric
series. We show that
a very well-poised balanced ${}_{12}{W}_{11}$ series is transformed to
a ${}_{4}{\phi}_{3}$ series which is
neither balanced
nor well-poised.
Recall the following proposition:
\begin{prp}[\protect{\cite[Proposition 7.3]{NS}}]
We have for $r, \theta \in \mathbb{Z}_{\geq 0}$
\begin{gather}
{}_{r+9}W_{r+8}\biggl(a;q^{-\theta},q^{\theta}af,a_1,\dots,a_r,
\biggl(\frac{aq}{f}\biggr)^{{1 \over 2}},
-\biggl(\frac{aq}{f}\biggr)^{{1 \over 2}}, \biggl(\frac{aq^2}{f}\biggr)^{{1 \over 2}},
-\biggl(\frac{aq^2}{f}\biggr)^{{1 \over 2}};q,z \biggr)\label{q} \\
\quad{}=
{(aq,f^2/q;q)_{\theta} \over(af,f;q)_{\theta} }
 \sum_{m\geq 0}
{(q/f,q^{-\theta},aq/f;q)_m \over (q,q^{-\theta}q^2/f^2,aq;q)_m}
q^m\,
{}_{r+5}W_{r+4}
\big(a;q^{-m},q^{m}aq/f,a_1,\dots,a_r;q,z\big).\nonumber
\end{gather}
\end{prp}

The main result in this section is as follows:

{\samepage

\begin{thm}\label{koushiki-2} Assume $af =  a_2 a_3$, then
\begin{gather}
{}_{12}W_{11}\biggl( a ; q^{-\theta}, q^{\theta}af, f, a_2, a_3, \biggl(\frac{aq}{f}\biggr)^{{1 \over 2}},
-\biggl(\frac{aq}{f}\biggr)^{{1 \over 2}}, \biggl(\frac{aq^2}{f}\biggr)^{{1 \over 2}},
-\biggl(\frac{aq^2}{f}\biggr)^{{1 \over 2}} ; q, q/f \biggr) \nonumber \\
\qquad {}=
{ (aq, af/a_2 ;q)_{\theta} \over (af, aq/a_2; q)_{\theta} }
{}_{4}{\phi}_{3}\biggl[
{{ q^{-{\theta}}, q^{-{\theta}}a_2/a, f, a_2}\atop{ q^{-{\theta+1}}a_2/af,
q^{-{\theta+1}}/f, aq/a_3 }} ;q, q^2/f^2 \biggr]. \label{sono1}
\end{gather}
\end{thm}}

\begin{rem} By one of the anonymous referees,
it was pointed out that~(\ref{sono1}) is  a~special case of the
transformation formula obtained by Langer, Schlosser and Warnaar~\cite[equation~(4.2)]{LSW}. Namely,
we have~(\ref{sono1}) by letting $d\rightarrow 0$ (or~$d\rightarrow \infty$) in the $p=0$ case of
\cite[equation~(4.2)]{LSW}. The authors thank the referee for informing them of this fact.
\end{rem}

\begin{rem}
The ${}_{4}{\phi}_{3}$ series in the right hand side of (\ref{sono1}) is neither balanced
nor well-poised. However it has the following structure:
set for simplicity $u_1 := q^{-\theta}$, $u_2 := q^{-\theta}a_2/a$,
$u_3 := f$, $u_4 := a_2$, $v_1 := q$, $v_2 := q^{-\theta+1}a_2/af$,
$v_3 := q^{-\theta+1}/f$ and $v_4 := aq/a_3$. Then we have
\begin{itemize}\itemsep=0pt
\item[(i)] $u_1 v_1 = u_3 v_3 = q^{-\theta +1}$ and
$u_2 v_4 = u_4 v_2=q^{-\theta +1} a_2/a_3$,
\item[(ii)] $u_1 v_4 = u_4 v_3 = a q^{-\theta +1}/a_3$ and
$u_2 v_1 = u_3 v_2 =  q^{-\theta +1} a_2/a$.
\end{itemize}
\end{rem}

\begin{proof}[Proof of Theorem \ref{koushiki-2}]
Applying (\ref{q}) to the left hand side of (\ref{sono1}), we have
\begin{gather}
\mbox{l.h.s.\ of (\ref{sono1})}=
{(aq, f^2/q;q)_{\theta} \over (af, f;q)_{\theta}}\nonumber\\
\hphantom{\mbox{l.h.s.\ of (\ref{sono1})}=}{}\times
\sum_{m \geq 0}{(q/f, q^{-\theta}, aq/f;q)_m \over
(q, q^{-\theta+2}/f^2, aq;q)_m} q^m\,
{}_{8}W_{7}\big( a ; q^{-m}, aq^{m+1}/f, f, a_2,
a_3 ; q, q/f \big).\!\!\! \label{sono2}
\end{gather}
We apply Watson's formula \cite[p.~35, equation~(2.5.1)]{GR}
\begin{gather}
 {}_{8}{\phi}_{7}\biggl[{ { a,qa^{1 \over 2}, -qa^{1 \over 2}, b, c, d, e, q^{-n} }\atop
{ a^{1 \over 2}, -a^{1 \over 2}, aq/b, aq/c, aq/d, aq/e, aq^{n+1} }} ; q, {a^2q^{2+n} \over
bcde} \biggr] \nonumber \\
\qquad {} =  { (aq, aq/{de} ;q)_n \over (aq/d, aq/e; q)_n }
{}_{4}{\phi}_{3}\biggl[
{{ q^{-n}, d, e, aq/{bc}}\atop{ aq/b, aq/c, deq^{-n}/a}} ;q, q \biggr],\label{Watson}
\end{gather}
to the right hand side of (\ref{sono2}) with the substitutions
$b = aq^{m+1}/f$, $c = a_3$, $d = a_2$ and $e = f$. Then we have
\begin{gather}
 \mbox{r.h.s.\ of (\ref{sono2})}
 =
{(aq, f^2/q;q)_{\theta} \over (af, f;q)_{\theta}}
\sum_{m \geq 0}{(q^{-\theta}, aq/a_2 f, q/f ; q)_m \over
(q^{-\theta+2}/f^2, aq/a_2, q ; q)_m} q^m\nonumber\\
\hphantom{\mbox{r.h.s.\ of (\ref{sono2})}=}{}
\times
\sum_{j=0}^{m}{(f, a_2, q^{-m}f/a_3, q^{-m};q)_{j} \over (aq/a_3, q^{-m}a_2f/a, q^{-m}f, q;q)_{j}} q^j. \label{sono2.75}
\end{gather}
Now we need to change the order of the summation.
Setting $s := m - j$ and using the condition $af = a_2a_3$, we have
\begin{gather}
 \mbox{r.h.s.\ of (\ref{sono2.75})}
 = {(aq, f^2/q;q)_{\theta} \over (af, f;q)_{\theta}} \sum_{j=0}^{\theta}
{(q^{-\theta}, f, a_2;q)_{j} \over (q^{-\theta + 2}/f^2, aq/a_3, q;q)_{j}} (q/f)^{2j}\nonumber\\
\hphantom{\mbox{r.h.s.\ of (\ref{sono2.75})}=}{}
\times
\sum_{s=0}^{\theta-j}
{(aq/a_2f, q/f, q^{-\theta + j};q)_{s} \over (aq/a_2, q, q^{-\theta + j + 2}/f^2;q)_{s}} q^s. \label{sono3}
\end{gather}
As a f\/inal step,
we apply the $q$-Saalsch\"utz transformation formula \cite[p.~13, equa\-tion~(1.7.2)]{GR} to the
 summation with respect to~$s$
of~(\ref{sono3}). Then we have
\begin{gather*}
\mbox{r.h.s.\ of (\ref{sono3})}=
{(aq, q;q)_{\theta} \over (af, f;q)_{\theta}}\sum_{j=0}^{\theta}
{(af/a_2, f;q)_{\theta - j} \over (aq/a_2, q;q)_{\theta - j}}{(f, a_2;q)_j \over (q, aq/a_3;q)_j}
\nonumber\\
\hphantom{\mbox{r.h.s.\ of (\ref{sono3})}}{}
=
{ (aq, af/a_2 ;q)_{\theta} \over (af, aq/a_2; q)_{\theta} }
{}_{4}{\phi}_{3}\biggl[
{{ q^{-{\theta}}, q^{-{\theta}}a_2/a, f, a_2}\atop{ q^{-{\theta+1}}a_2/af,
q^{-{\theta+1}}/f, aq/a_3 }} ;q, q^2/f^2 \biggr].
\end{gather*}
This completes the proof of Theorem~\ref{koushiki-2}.
\end{proof}

In what follows, we use  Theorem~\ref{koushiki-2} in  the form
\begin{gather}
 {}_{12}W_{11}
\biggl( a ; q^{-\theta}, q^{\theta}af, f, a_2, a_3, \biggl(\frac{aq}{f}\biggr)^{{1 \over 2}},
-\biggl(\frac{aq}{f}\biggr)^{{1 \over 2}}, \biggl(\frac{aq^2}{f}\biggr)^{{1 \over 2}},
-\biggl(\frac{aq^2}{f}\biggr)^{{1 \over 2}} ; q, q/f \biggr) \nonumber \\
\qquad{} =
{(aq, q;q)_{\theta} \over (af, f;q)_{\theta}}\sum_{j=0}^{\theta}
{(af/a_2, f;q)_{\theta - j} \over (aq/a_2, q;q)_{\theta - j}}{(f, a_2;q)_j \over (q, aq/a_3;q)_j}.\label{hitotsumaenoprp}
\end{gather}

\section{Transformation formula II}\label{section3}

In this section, we present a transformation formula
which will be used
to describe the Macdonald polynomials of types
$D_n$ and $C_n$.

\begin{thm} \label{N=ippan}
Let $n \in \mathbb{Z}_{\geq 2}$. Fix $K,  m_1, m_2, \dots, m_n \in \mathbb{Z}_{\geq 0}$ arbitrarily.
Set $ m_{l,n}:=\sum\limits_{k=l}^{n} m_k $, $ \phi_{l,n}:=\sum\limits_{k=l}^{n} \phi_k $ for simplicity of display. We have
\begin{gather}
\sum_{\phi_1 ,\phi_2 , \dots , \phi_{n-1} , i\geq 0\atop
\phi_1 + \phi_2 + \cdots + \phi_{n-1} + i = K}
\Biggl(
\prod_{1 \leq l \leq n-1}
{(t;q)_{\phi_l} (t;q)_{\phi_l + m_l}
\over (q;q)_{\phi_l} (q;q)_{\phi_l +m_l}}  \nonumber\\
\qquad\quad{}\times
\frac{
(t^{n-l-1} q^{\phi_l + 2\phi_{l+1,n-1} + m_{l,n}+1 } ;q)_{\phi_l}
(t^{n-l} q^{2\phi_{l+1,n-1} + m_{l+1,n}} ;q)_{\phi_l}}
{
(t^{n-l} q^{\phi_l + 2\phi_{l+1,n-1} + m_{l,n} } ;q)_{\phi_l}
(t^{n-l-1} q^{2\phi_{l+1,n-1} + m_{l+1,n} + 1} ;q)_{\phi_l}}
\Biggr)\nonumber\\
 \qquad\quad{}  \times
{ (t;q)_{m_n}
\over (q;q)_{m_n}}
\frac{(t;q)_{i}(t^n q^{2K+m_{1,n}-2i};q )_i}
{(q;q)_{i}(t^{n-1} q^{2K+m_{1,n}-2i +1};q)_i}
\nonumber \\
\qquad{}= \sum_{\phi_1 ,\phi_2 , \dots , \phi_{n-1} , \phi_{n}\geq 0\atop \phi_1 + \phi_2 + \cdots + \phi_{n} = K}
\prod_{1 \leq j \leq n}{(t;q)_{\phi_j} (t;q)_{\phi_j + m_j}
\over (q;q)_{\phi_j} (q;q)_{\phi_j +m_j}}.
\label{sono6}
\end{gather}
\end{thm}

We prove Theorem~\ref{N=ippan} by induction on~$n$.
In order to clarify  the structure of our proof, we f\/irst conf\/irm the case $n=2$
in Section~\ref{n=2}, and then  treat the general case in Section~\ref{general}.

\subsection[The case $n=2$]{The case $\boldsymbol{n=2}$}\label{n=2}

\begin{prp}\label{koushiki-3}
Fix $K,m_1, m_2 \in \mathbb{Z}_{\geq 0}$ arbitrarily.
We have
\begin{gather}
 \sum_{\phi_1,i\geq 0\atop \phi_1+i=K}
{(t;q)_{m_1+\phi_1} (t;q)_{m_2}
(t,q^{m_1+m_2+\phi_1+1},t q^{m_2};q)
_{\phi_1}
(t,t^2q^{2K+m_1+m_2-2i};q)_i
\over
(q;q)_{m_1+\phi_1}(q;q)_{m_2}
(q,t q^{m_1+m_2+\phi_1},
q^{m_2+1};q)_{\phi_1}
(q,tq^{2K+m_1+m_2-2i+1};q)_i
}
\nonumber\\
\qquad {}=
 \sum_{\phi_1,\phi_2\geq 0\atop \phi_1+\phi_2=K}
 {(t;q)_{m_1+\phi_1} (t;q)_{m_2+\phi_2}
 (t;q)_{\phi_1} (t;q)_{\phi_2}
\over
 (q;q)_{m_1+\phi_1} (q;q)_{m_2+\phi_2}
 (q;q)_{\phi_1} (q;q)_{\phi_2} }.\label{N=2}
\end{gather}
\end{prp}

\begin{proof} One f\/inds that the summation in l.h.s.\ of~(\ref{N=2}) with respect to $\phi_1$
is given by the following~${}_{12}W_{11}$ series
\begin{gather}
\mbox{l.h.s.\ of (\ref{N=2})}
 =
{(t,t^2 q^{m_1 + m_2};q)_K (t;q)_{m_1} (t;q)_{m_2}\over
{(q, t q^{m_1 + m_2 +1};q)_K (q;q)_{m_1} (q;q)_{m_2}}} \nonumber \\
\qquad{}
\times
{}_{12}W_{11}\biggl(a ; q^{-K},q^K a f, f, b,
c, \biggl({aq\over f}\biggr)^{1\over 2},- \biggl({aq\over f}\biggr)^{1\over 2},
\biggl({aq^2 \over f}\biggr)^{1 \over 2},-\biggl({aq^2 \over f}\biggr)^{1 \over 2} ; q, q/f \biggr),\label{sono5}
\end{gather}
where
$a=tq^{m_1+m_2}$, $b=t q^{m_1}$, $c=t q^{m_2}$, $f=t$.
Then applying formula~(\ref{hitotsumaenoprp}), the r.h.s.\ of~(\ref{sono5}) is rewritten as
\begin{gather*}
 {(t;q)_{m_1} (t;q)_{m_2}\over
{(q;q)_{m_1} (q;q)_{m_2}}}
\sum_{\phi_1=0}^{K}
{(tq^{m_2}, t;q)_{K - \phi_1} \over (q^{m_2 +1}, q;q)_{K - \phi_1}}{(t, t q^{m_1};q)_{\phi_1} \over (q, q^{m_1 +1};q)_{\phi_1}} \\
 \qquad{}=
 \sum_{\phi_1=0}^{K}
 {(t;q)_{m_1+\phi_1} (t;q)_{m_2+K-\phi_1}
 (t;q)_{\phi_1} (t;q)_{K-\phi_1}
\over
 (q;q)_{m_1+\phi_1} (q;q)_{m_2+K-\phi_1}
 (q;q)_{\phi_1} (q;q)_{K-\phi_1} }.
\end{gather*}
This completes the proof of Proposition~\ref{koushiki-3}.
\end{proof}

\subsection{The general case}\label{general}

Assume the validity of the transformation formula~(\ref{sono6}) for $n-1$.
We have
\begin{gather}
\mbox{l.h.s.\ of (\ref{sono6})}
 =
\sum_{\phi_{n-1} = 0}^{K} \sum_{\phi_{n-2} = 0}^{K -\phi_{n-1}}
\cdots \sum_{\phi_{2} = 0}^{K -\phi_{3,n-1}}
\prod_{2 \leq l \leq n-1}{(t;q)_{\phi_{l}} (t;q)_{\phi_{l} + m_l}
\over (q;q)_{\phi_{l}} (q;q)_{\phi_{l} + m_l}} \nonumber\\
\hphantom{\mbox{l.h.s.\ of (\ref{sono6})}=}{}\times
\frac{
(t^{n-l-1} q^{\phi_l + 2\phi_{l+1,n-1} + m_{l,n} + 1} ;q)_{\phi_{l}}
(t^{n-l} q^{2\phi_{l+1,n-1} + m_{l+1,n}} ;q)_{\phi_l}
}
{
(t^{n-l} q^{\phi_l + 2\phi_{l+1,n-1} + m_{l,n}} ;q)_{\phi_l}
(t^{n-l-1} q^{2\phi_{l+1,n-1} + m_{l+1,n} + 1} ;q)_{\phi_l}
} \nonumber\\
\hphantom{\mbox{l.h.s.\ of (\ref{sono6})}=}{}\times
\sum_{\phi_1 = 0}^{K -\phi_{2,n-1}}
{(t;q)_{\phi_1} (t;q)_{\phi_1 + m_1} (t;q)_{m_n}
(t,t^n q^{2\phi_{1,n-1} + m_{1,n}};q )_{K -\phi_{1,n-1}} \over
(q;q)_{\phi_1} (q;q)_{\phi_1 + m_1} (q;q)_{m_n}
(q,t^{n-1} q^{2\phi_{1,n-1} + m_{1,n} +1};q)_{K -\phi_{1,n-1}}}
\nonumber \\
\hphantom{\mbox{l.h.s.\ of (\ref{sono6})}=}{}\times
\frac{
(t^{n-2} q^{\phi_1 + 2\phi_{2,n-1} + m_{1,n} + 1} ;q)_{\phi_1}
(t^{n-1} q^{2\phi_{2,n-1} + m_{2,n}} ;q)_{\phi_1}
}
{
(t^{n-1} q^{\phi_1 + 2\phi_{2,n-1} + m_{1,n}} ;q)_{\phi_1}
(t^{n-2} q^{2\phi_{2,n-1} + m_{2,n} + 1} ;q)_{\phi_{1}}
}. \label{sono7}
\end{gather}
The summation  with respect to $\phi_1$ in (\ref{sono7}) can be written as follows
\begin{gather}
 {(t;q)_{m_1} (t;q)_{m_n} (t,
t^nq^{2\phi_{2,n-1} + m_{1,n}};q)_{K-\phi_{2,n-1}} \over
(q;q)_{m_1} (q;q)_{m_n} (q,
t^{n-1}q^{2\phi_{2,n-1} + m_{1,n} +1};q)_{K-\phi_{2,n-1}}} \nonumber \\
\qquad{} \times {}_{12}W_{11}\biggl( a ; q^{-\theta}, q^{\theta}af, f, b, c, \biggl(\frac{aq}{f}\biggr)^{{1 \over 2}},
-\biggl(\frac{aq}{f}\biggr)^{{1 \over 2}}, \biggl(\frac{aq^2}{f}\biggr)^{{1 \over 2}},
-\biggl(\frac{aq^2}{f}\biggr)^{{1 \over 2}} ; q, q/f \biggr), \label{sono8}
\end{gather}
where
$
 a=t^{n-1}q^{2\phi_{2,n-1} + m_{1,n}}, \,  b=tq^{m_1}, c=t^{n-1}q^{2\phi_{2,n-1}
+ m_{2,n}}, \,f=t, \, \theta = K - \phi_{2,n-1}. $
Applying the formula~(\ref{hitotsumaenoprp}),  (\ref{sono8}) is transformed into
\begin{gather}
{(t;q)_{m_1} (t;q)_{m_n} \over (q;q)_{m_1} (q;q)_{m_n}}\nonumber\\
\qquad{}\times
\sum_{j = 0}^{K - \phi_{2,n-1}}
{(t^{n-1}q^{2\phi_{2,n-1} + m_{2,n}};q)_{K -j - \phi_{2,n-1}}(t;q)_{K -j - \phi_{2,n-1}} (t;q)_j(tq^{m_1};q)_j
\over
(t^{n-2}q^{2\phi_{2,n-1} + m_{2,n} +1};q)_{K -j - \phi_{2,n-1}}(q;q)_{K - j - \phi_{2,n-1}} (q;q)_j(q^{m_1 +1};q)_j}.   \label{sono9}
\end{gather}
Using (\ref{sono9})  and changing the order of the summations,  we can
express the r.h.s.\ of~(\ref{sono7}) as follows
\begin{gather}
 \sum_{j = 0}^{K} {(t;q)_j (t;q)_{j+m_1} \over (q;q)_j (q;q)_{j+m_1}}
\sum_{\phi_{n-1} = 0}^{K -j} \sum_{\phi_{n-2} = 0}^{K -j - \phi_{n-1}}
\cdots \sum_{\phi_3 = 0}^{K -j - \phi_{4,n-1}}
\prod_{3 \leq l \leq n-1}{(t;q)_{\phi_l} (t;q)_{\phi_l + m_l}
\over (q;q)_{\phi_l} (q;q)_{\phi_l + m_l}} \nonumber\\
{}\times
\frac{
(t^{n-l-1} q^{\phi_l + 2\phi_{l+1,n-1} + m_{l,n} + 1} ;q)_{\phi_l}
(t^{n-l} q^{2\phi_{l+1,n-1} + m_{l+1,n}} ;q)_{\phi_l}
}
{
(t^{n-l} q^{\phi_l + 2\phi_{l+1,n-1} + \phi_{l,n}} ;q)_{\phi_l}
(t^{n-l-1} q^{2\phi_{l+1,n-1} + m_{l+1,n} + 1} ;q)_{\phi_l}
} \nonumber\\
{} \times
\sum_{\phi_2 = 0}^{K - j -\phi_{3,n-1}}\!\!
{(t;q)_{\phi_2 + m_2} (t;q)_{m_n}(t;q)_{\phi_2}
(t^{n-1}q^{2\phi_{2,n-1}
 + m_{2,n-1}};q)_{K -j-\phi_{2,n-1}}(t;q)_{K -j-\phi_{2,n-1}}
\over
(q;q)_{\phi_2 + m_2}(q;q)_{m_n}(q;q)_{\phi_2}
(t^{n-2}q^{2\phi_{2,n-1}
 + m_{2,n} +1};q)_{K -j-\phi_{2,n-1}}(q;q)_{K -j-\phi_{2,n-1}}} \nonumber \\
{} \times
{(t^{n-2}q^{2\phi_{3,n-1}
 + m_{3,n}};q)_{\phi_2}
(t^{n-3}q^{\phi_2 + 2\phi_{3,n-1}
 + m_{2,n} +1};q)_{\phi_2} \over
(t^{n-3}q^{2\phi_{3,n-1}
 + m_{3,n} +1};q)_{\phi_2}
( t^{n-2}q^{\phi_2 + 2\phi_{3,n-1}
 + m_{2,n}};q)_{\phi_2}}. \label{sono10}
\end{gather}
By the induction hypothesis, (\ref{sono10}) can be rewritten as
\begin{gather*}
 \sum_{j = 0}^{K} {(t;q)_j (t;q)_{j+m_1} \over (q;q)_j (q;q)_{j+m_1}}
\sum_{\phi_2 + \phi_3 + \cdots + \phi_{n-1} + \phi= K -j}
{(t;q)_{\phi} (t;q)_{\phi + m_n} \over (q;q)_{\phi} (q;q)_{\phi + m_n}}
\prod_{2 \leq l \leq n-1}
\frac{(t;q)_{\phi_l}(t;q)_{\phi_l + m_l}}
{(q;q)_{\phi_l}(q;q)_{\phi_l + m_l}} \\
\qquad{}=  \sum_{j  + \phi_2 + \phi_3 + \cdots + \phi_{n-1} +\phi = K}
{(t;q)_j (t;q)_{j+m_1} \over (q;q)_j (q;q)_{j+m_1}}
{(t;q)_{\phi} (t;q)_{\phi + m_n} \over (q;q)_{\phi} (q;q)_{\phi + m_n}}
\prod_{2 \leq l \leq n-1}
\frac{(t;q)_{\phi_l}(t;q)_{\phi_l + m_l}}
{(q;q)_{\phi_l}(q;q)_{\phi_l + m_l}} \\
\qquad{}=    \mbox{r.h.s.\ of (\ref{sono6})}.
\end{gather*}
Hence we have completed the proof of Theorem \ref{N=ippan}.

\section[Tableau formulas for Macdonald polynomials of type $D_n$]{Tableau formulas for Macdonald polynomials of type $\boldsymbol{D_n}$}\label{section4}

In this section, we investigate the tableau formula
for the one-row Macdonald polynomials
of type~$D_n$.
Let $I := \{ 1,2,\dots,n-1,n, \overline{n}, \overline{n-1}, \dots, \overline{1} \}$
be the index set with the ordering
\begin{gather}\label{OrderOfD}
1 \prec 2 \prec \cdots \prec n-1{\raisebox{-3pt}{\rotatebox{30}{$\prec$}}\atop \raisebox{5pt}{\rotatebox{-30}{$\prec$}}}\,\,
 {\raisebox{3pt}{$n$}\atop\raisebox{-2pt}{$\overline{n}$}}\,
{\raisebox{3pt}{\rotatebox{-30}{$\prec$}}\atop \raisebox{-1pt}{\rotatebox{30}{$\prec$}}} \,\,
\overline{n-1}
\prec \cdots \prec \overline{1}.
\end{gather}
Denoting by $\Lambda_1$ the f\/irst fundamental weight of type~$D_n$, let $P_{(r)}^{(D_n)}(x;q,t)$ be the Macdonald
polynomials of type $D_n$ associated with the weights $r \Lambda_1$ for
$r \in {\mathbb{ Z}_{\geq 0}}$.

We recall Lassalle's formula
for $P_{(r)}^{(D_n)}(x;q,t)$.
Lassalle introduced
$G_r(x;q,t)$ def\/ined by the generating function
\begin{gather}\label{G(x;q,t)}
\prod_{i=1}^{n}\frac{(tux_i;q)_{\infty}}{(ux_i;q)_{\infty}}
\frac{(tu/x_i;q)_{\infty}}{(u/x_i;q)_{\infty}} = \sum_{r \geq 0} G_r(x;q,t)u^r.
\end{gather}
Comparing the coef\/f\/icient of $u^r$ of the equation (\ref{G(x;q,t)}), we obtain
\begin{gather}
G_r(x;q,t) =
\sum_{
\theta_{1}  + \theta_{2} + \cdots + \theta_{\overline{1}} = r}
\prod_{i \in I}
\frac{(t;q)_{\theta_{i}}}{(q;q)_{\theta_{i}}}
x_1^{\theta_1 - \theta_{\overline{1}}}x_2^{\theta_2 - \theta_{\overline{2}}}\cdots
x_n^{\theta_n - \theta_{\overline{n}}}, \label{Lassalle-G}
\end{gather}
where $\theta_i$, $\theta_{\overline{i}} \in \mathbb{Z}_{\geq 0}$, $i = 1,2, \dots, n$.

The following theorem \cite[Theorem~5.2]{HNS} was conjectured by Lassalle~\cite{La}.

\begin{thm}[\cite{HNS,La}] \label{Lassalle} For any positive integer $r$ we have
\begin{gather}
G_r(x;q,t) = \sum^{[r/2]}_{i = 0}
\frac{(t;q)_{r-2i}}{(q;q)_{r-2i}}P_{(r-2i)}^{(D_n)}(x;q,t)
\frac{(t;q)_i(t^n q^{r-2i};q )_i}{(q;q)_i(t^{n-1} q^{r-2i +1};q)_i}.\label{G_to_P}
\end{gather}
Conversely
\begin{gather}
 P_{(r)}^{(D_n)}(x;q,t) = {(q;q)_r \over (t;q)_r}
\sum^{[r/2]}_{i = 0}
G_{r-2i}(x;q,t)t^i
\frac{(1/t;q)_i(t^n q^{r-i};q )_i}{(q;q)_i(t^{n-1} q^{r-i};q)_i}{1-t^n q^{r-2i} \over 1-t^n q^{r-i}}. \label{Lassalle-PD}
\end{gather}
\end{thm}

\begin{rem}\sloppy
If we insert~(\ref{Lassalle-G}) in~(\ref{Lassalle-PD}), we have an explicit combinatorial formula for
$P_{(r)}^{(D_n)}(x;q,t)$. However, it is not clear how we can extract
the combinatorics of the Kashiwara--Nakashima tableaux of type~$D$ from this Lassalle's version.
\end{rem}

Here, we establish the tableau formula for $P^{(D_n)}_{(r)}(x;q,t)$.

\begin{thm} We have\label{main}
\begin{gather}
 P_{(r)}^{(D_n)}(x;q,t)=
\frac{(q;q)_r}{(t;q)_r}\sum_{
\theta_1 + \theta_2 + \cdots +\theta_{\overline{1}} = r \atop
\theta_n\theta_{\overline{n}}=0}
\prod_{k \in I}
\frac{(t;q)_{\theta_{k}}}{(q;q)_{\theta_{k}}}\label{TSF_D} \\
{}\times \prod_{1 \leq l \leq n-1}\!\!
\frac{(t^{n-l} q^{\theta_{l+1} + \theta_{l+2} + \cdots
 + \theta_{\overline{l+1}}} ;q)_{\theta_{\overline{l}}}
(t^{n-l-1} q^{\theta_{l} + \theta_{l+1} + \cdots
 + \theta_{\overline{l+1}} +1 } ;q)_{\theta_{\overline{l}}}}
{(t^{n-l-1} q^{\theta_{l+1} + \theta_{l+2} + \cdots + \theta_{\overline{l+1}} + 1} ;q)_{\theta_{\overline{l}}}
(t^{n-l} q^{\theta_{l} + \theta_{l+1} + \cdots + \theta_{\overline{l+1}} } ;q)_{\theta_{\overline{l}}}}
x_1^{\theta_1 - \theta_{\overline{1}}}x_2^{\theta_2 - \theta_{\overline{2}}}\cdots
x_n^{\theta_n - \theta_{\overline{n}}}.\nonumber
\end{gather}
\end{thm}

\begin{rem}\label{rem_D}
Set
$X:=t^{n-l} q^{\theta_{l+1} + \theta_{l+2} + \cdots
 + \theta_{\overline{l+1}}} $ and $Y:=t^{n-l-1} q^{\theta_{l+1} + \theta_{l+2} + \cdots
 + \theta_{\overline{l+1}}} $ for simplicity. The last product in (\ref{TSF_D}) can be rewritten as
\begin{gather*}
\prod_{1 \leq l \leq n-1}
\frac{(X;q)_{\theta_{\overline{l}}}(q^{\theta_l}Y;q)_{\theta_{\overline{l}}}}
{(Y;q)_{\theta_{\overline{l}}}(q^{\theta_l}X;q)_{\theta_{\overline{l}}}}
=
\prod_{1 \leq l \leq n-1}
\frac{(X;q)_{\theta_l}(X;q)_{\theta_{\overline{l}}}(Y;q)_{\theta_l+ \theta_{\overline{l}}}}
{(Y;q)_{\theta_l}(Y;q)_{\theta_{\overline{l}}}(X;q)_{\theta_l + \theta_{\overline{l}}}}.
\end{gather*}
This implies that r.h.s.\ of (\ref{TSF_D})
has the symmetry $(\mathbb{Z}/2\mathbb{Z})^n$. Namely,
it is invariant under the exchange
$x_l \leftrightarrow {1 \over x_l}$, $1 \leq l \leq n-1$.
\end{rem}

\begin{rem}
It would be an intriguing problem to show the factorization of the Macdonald polynomial $P_{(r)}^{(D_n)}(x;q,t)$
from our formula~(\ref{TSF_D})
when we  make the principal specialization:
\begin{gather*}
P_{(r)}^{(D_n)}\big(t^{n-1},\ldots,t,1;q,t\big)=t^{-r(n-1)}{(t^n;q)_r (t^{2(n-1)};q)_r \over (t;q)_r (t^{(n-1)};q)_r}.
\end{gather*}
\end{rem}

\begin{rem}
Setting $n=1$ in (\ref{TSF_D}), we have
\begin{gather*}
P_{(r)}^{(D_1)}(x;q,t)
=x^r + x^{-r}.
\end{gather*}
Setting $n=2$ in (\ref{TSF_D}), we have
\begin{gather*}
P_{(r)}^{(D_2)}(x;q,t) =  \frac{(q;q)_r}{(t;q)_r}
\sum_{\theta_1 + \theta_2 + \theta_{\overline{2}} + \theta_{\overline{1}} = r\atop
\theta_2  \theta_{\overline{2}} =0}
{(t;q)_{\theta_1} \over  (q;q)_{\theta_1}}
{(t;q)_{\theta_2} \over  (q;q)_{\theta_2}}
{(t;q)_{\theta_{\overline{2}}} \over  (q;q)_{\theta_{\overline{2}}}}
{(t;q)_{\theta_{\overline{1}}} \over  (q;q)_{\theta_{\overline{1}}}} \nonumber \\
\hphantom{P_{(r)}^{(D_2)}(x;q,t) =}{}
\times
{
(t q^{\theta_2+\theta_{\overline{2}}};q)_{\theta_{\overline{1}}}
\over
(q q^{\theta_2+\theta_{\overline{2}}};q)_{\theta_{\overline{1}}}
}
{
(q q^{\theta_1+\theta_2+\theta_{\overline{2}}};q)_{\theta_{\overline{1}}}
\over
(t q^{\theta_1+\theta_2+\theta_{\overline{2}}};q)_{\theta_{\overline{1}}}
}
  x_1^{\theta_1 - \theta_{\overline{1}}} x_2^{\theta_2 - \theta_{\overline{2}}} \nonumber\\
\hphantom{P_{(r)}^{(D_2)}(x;q,t)}{} =
\left(  \frac{(q;q)_r}{(t;q)_r}
  \sum_{\mu_1 + \mu_2 = r}
  {(t;q)_{\mu_1} \over  (q;q)_{\mu_1}}
{(t;q)_{\mu_2} \over  (q;q)_{\mu_2}}
x_1^{(\mu_1-\mu_2)/2}x_2^{-(\mu_1-\mu_2)/2}\right)\\
\hphantom{P_{(r)}^{(D_2)}(x;q,t) =}{}\times
\left(  \frac{(q;q)_r}{(t;q)_r}
  \sum_{\nu_1 + \nu_2 = r}
  {(t;q)_{\nu_1} \over  (q;q)_{\nu_1}}
{(t;q)_{\nu_2} \over  (q;q)_{\nu_2}}
x_1^{(\nu_1-\nu_2)/2}x_2^{(\nu_1-\nu_2)/2}\right),\nonumber
\end{gather*}
which shows the symmetry $D_2=A_1\times A_1$.
\end{rem}

\begin{proof}[Proof of Theorem \ref{main}]
Let $\big\{\Psi^{(D_n)}_{(r)}(x;q,t)\big\}_{r\in \mathbb{Z}_{\geq 0}}$ be a certain collection of  Laurent polyno\-mials.
By using Lassalle's formula~(\ref{G_to_P}),
it is easily proved by induction that the inf\/inite system of equalities for $\Psi^{(D_n)}_{(r)}(x;q,t)${\samepage
\begin{gather}
  G_r(x;q,t) = \sum^{[r/2]}_{i = 0}
\frac{(t;q)_{r-2i}}{(q;q)_{r-2i}}\Psi_{(r-2i)}^{(D_n)}(x;q,t)
\frac{(t;q)_i(t^n q^{r-2i};q )_i}{(q;q)_i(t^{n-1} q^{r-2i +1};q)_i},\qquad  r\in \mathbb{Z}_{\geq 0},\label{G_to_Psi}
\end{gather}
gives us $\Psi_{(r)}^{(D_n)}(x;q,t)=P_{(r)}^{(D_n)}(x;q,t)$, $r\in \mathbb{Z}_{\geq 0}$.}

Set
\begin{gather*}
 \Psi^{(D_n)}_{(r)}(x;q,t)=\frac{(q;q)_r}{(t;q)_r}\sum_{
\theta_1 + \theta_2 + \cdots +\theta_{\overline{1}} = r \atop
\theta_n\theta_{\overline{n}}=0}
\prod_{k \in I}
\frac{(t;q)_{\theta_{k}}}{(q;q)_{\theta_{k}}}   \\
\hphantom{\Psi^{(D_n)}_{(r)}(x;q,t)=}{}
 \times \prod_{1 \leq l \leq n-1}
\frac{(t^{n-l} q^{\theta_{l+1} + \theta_{l+2} + \cdots
 + \theta_{\overline{l+1}}} ;q)_{\theta_{\overline{l}}}
(t^{n-l-1} q^{\theta_{l} + \theta_{l+1} + \cdots
 + \theta_{\overline{l+1}} +1 } ;q)_{\theta_{\overline{l}}}}
{(t^{n-l-1} q^{\theta_{l+1} + \theta_{l+2} + \cdots + \theta_{\overline{l+1}} + 1} ;q)_{\theta_{\overline{l}}}
(t^{n-l} q^{\theta_{l} + \theta_{l+1} + \cdots + \theta_{\overline{l+1}} } ;q)_{\theta_{\overline{l}}}}\\
\hphantom{\Psi^{(D_n)}_{(r)}(x;q,t)=}{}
\times x_1^{\theta_1 - \theta_{\overline{1}}}x_2^{\theta_2 - \theta_{\overline{2}}}\cdots
x_n^{\theta_n - \theta_{\overline{n}}}.
\end{gather*}
We prove this family of  Laurent polynomials satisf\/ies (\ref{G_to_Psi}).
In view of the  $(\mathbb{Z}/2\mathbb{Z})^n$ symmetry of $\Psi^{(D_n)}_{(r)}(x;q,t)$,
it is suf\/f\/icient to consider in  (\ref{G_to_Psi}) the
coef\/f\/icients of the monomials $x_1^{m_1}\cdots x_n^{m_n}$
with nonnegative powers $m_1,\dots,m_n \in \mathbb{Z}_{\geq 0}$ only.
Let $r\in \mathbb{Z}_{\geq 0}$, and
f\/ix $m_1,\dots,m_n \in \mathbb{Z}_{\geq 0}$ arbitrarily. Set  $K := {1 \over 2}(r - m_1 - m_2 - \cdots - m_n)$ for simplicity.
Setting
\begin{gather*}
 \theta_k  = m_k +\phi_k, \qquad \theta_{\overline{k}} =\phi_k,\qquad  1 \leq k \leq n-1,\qquad
 \theta_n =m_n,\qquad  \theta_{\overline{n}} = 0,
\end{gather*}
one f\/inds that the
coef\/f\/icients of the monomials $x_1^{m_1}\cdots x_n^{m_n}$ in (\ref{G_to_Psi})  is exactly given by  l.h.s.\ of~(\ref{sono6}).
On the other hand, the coef\/f\/icients of the monomials $x_1^{m_1}\cdots x_n^{m_n}$ in $G_r(x;q,t) $ is clearly r.h.s.\ of~(\ref{sono6}).
Hence we have proved~(\ref{G_to_Psi}), which establishes the tableau formula $P^{(D_n)}_{(r)}(x;q,t)=\Psi^{(D_n)}_{(r)}(x;q,t)$.
\end{proof}

\section{Transformation formula~III}\label{section5}

In this section, we present a transformation formula to describe the Macdonald polynomials of
type~$C_n$.

\begin{thm}\label{mainthm_C}
Let $n \in \mathbb{Z}_{\geq 2}$.
Fix $K, m_1, m_2, \dots, m_n \in \mathbb{Z}_{\geq 0}$ arbitrarily.
Set
$ m_{l,n}:=\sum\limits_{k=l}^{n} m_k $, $\phi_{l,n}:=\sum\limits_{k=l}^{n} \phi_k $ for simplicity of display. We have
\begin{gather}
\sum_{
\phi_1 , \phi_2 , \dots , \phi_{n} ,i\geq 0 \atop
\phi_1 + \phi_2 + \cdots + \phi_{n} +i = K}
\prod_{1 \leq k \leq n}{(t;q)_{\phi_k} (t;q)_{\phi_k + m_k}
\over (q;q)_{\phi_k} (q;q)_{\phi_k +m_k}}\cdot
\big(t^2/q\big)^i \frac{(t^{-1}q;q)_i (t^nq^{2K + m_{1,n} -2i};q)_i}
{(q;q)_i (t^{n+1}q^{2K + m_{1,n} -2i};q)_i} \nonumber \\
\qquad\quad{}\times
\prod_{1\leq l \leq n}\frac{
(t^{n-l+1} q^{\phi_l + \phi_{l+1,n} + m_{l,n}};q)_{\phi_l}
(t^{n-l+2}q^{ 2\phi_{l+1,n} + m_{l+1,n} -1};q)_{\phi_l}}
{(t^{n-l+2}q^{\phi_l + \phi_{l+1,n} + m_{l,n} -1};q)_{\phi_l}
(t^{n-l+1}q^{ 2\phi_{l+1,n} + m_{l+1,n}};q)_{\phi_l}} \nonumber \\
\qquad{}= \sum_{
\phi_1 , \phi_2 , \dots , \phi_{n} \geq 0 \atop
\phi_1 + \phi_2 + \cdots + \phi_{n} = K}
\prod_{1 \leq j \leq n}{(t;q)_{\phi_j} (t;q)_{\phi_j + m_j}
\over (q;q)_{\phi_j} (q;q)_{\phi_j +m_j}}.\label{N=ippan_C}
\end{gather}
\end{thm}

We prove Theorem~\ref{mainthm_C} by induction on $n$.
In Section~\ref{n=2_C} we show Theorem~\ref{mainthm_C} in the case of $n = 2$ and in Section~\ref{general_C} we treat the general case.

\subsection[The case $n=2$]{The case $\boldsymbol{n=2}$}\label{n=2_C}
Setting $n=2$, we have
\begin{gather*}
 \text{r.h.s.\ of (\ref{N=ippan_C})}
=  \sum_{\phi_1=0}^{K}
\frac{(t;q)_{m_1+\phi_1}(t;q)_{\phi_1}(t;q)_{m_2+K-\phi_1}(t;q)_{K-\phi_1}
}    {(q;q)_{m_1+\phi_1}(q;q)_{\phi_1}(q;q)_{m_2+K-\phi_1}(q;q)_{K-\phi_1}
}\\
\hphantom{\text{r.h.s.\ of (\ref{N=ippan_C})}  }{}
=   \sum_{\phi_1=0}^{K}\frac{ (t;q)_{m_1}(t;q)_{m_2}
(tq^{m_1};q)_{\phi_1}(t;q)_{\phi_1}(tq^{m_2};q)_{K-\phi_1}(t;q)_{K-\phi_1} }
    { (q;q)_{m_1}(q;q)_{m_2}(q^{m_1+1};q)_{\phi_1}
    (q;q)_{\phi_1}(q^{m_2+1};q)_{K-\phi_1}
    (q;q)_{K-\phi_1} }.
\end{gather*}
Then we have
\begin{gather}
 \text{l.h.s.\ of (\ref{N=ippan_C})}
= \sum_{\phi_2=0}^{K}\sum_{\phi_1=0}^{K-\phi_2}
    \frac{(t;q)_{\phi_2}(t;q)_{\phi_2+m_{2}}(t;q)_{\phi_1}(t;q)_{\phi_1+m_{1}}
}
    {(q;q)_{\phi_2}(q;q)_{\phi_2+m_{2}}(q;q)_{\phi_1}(q;q)_{\phi_1+m_{1}}}\nonumber\\
\hphantom{\text{l.h.s.\ of (\ref{N=ippan_C})}= }{}\times
    \frac{(t^2q^{m_1+m_2+2\phi_2+\phi_1};q)_{\phi_1}
(t^3q^{m_2+2\phi_2-1};q)_{\phi_1}}
    {(t^3q^{m_1+m_2+2\phi_2+\phi_1-1};q)_{\phi_1}(t^2q^{m_2+2\phi_2};q)_{\phi_1}}
\nonumber \\
\hphantom{\text{l.h.s.\ of (\ref{N=ippan_C})} =}{}
    \times \frac{ (tq^{m_2+\phi_2};q)_{\phi_2}(t^2q^{-1};q)_{\phi_2} }
    { (t^2q^{m_2-1+\phi_2};q)_{\phi_2}(t;q)_{\phi_2} }
    (t^2/q)^{K-\phi_2-\phi_1}\nonumber\\
\hphantom{\text{l.h.s.\ of (\ref{N=ippan_C})} =}{}\times
    \frac{ (q/t;q)_{K-\phi_2-\phi_1}
(t^2q^{m_1+m_2+2\phi_2+2\phi_1};q)_{K-\phi_2-\phi_1} }
    { (q;q)_{K-\phi_2-\phi_1}(t^3q^{m_1+m_2+2\phi_2+2\phi_1};q)_{K-\phi_2-\phi_1} }
\nonumber \\
{}
= \sum_{\phi_2=0}^{K}
    \frac{(t;q)_{m_2+\phi_2}(t;q)_{\phi_2}(tq^{m_2+\phi_2};q)_{\phi_2}(t^2q^{-1};q)_{\phi_2}
    (q/t;q)_{K-\phi_2}(t^2q^{m_1+m_2+2\phi_2};q)_{K-\phi_2}}
    {(q;q)_{m_2+\phi_2}(q;q)_{\phi_2}(tq^{m_2-1+\phi_2};q)_{\phi_2}(t;q)_{\phi_2}
    (q;q)_{K-\phi_2}(t^3q^{m_1+m_2+2\phi_2};q)_{K-\phi_2}} \nonumber \\
\hphantom{\text{l.h.s.\ of (\ref{N=ippan_C})}=}{}
\times  {}_8\phi_7\left[\begin{array}{@{}c@{}} a,qa^{\frac{1}{2}},-qa^{\frac{1}{2}},b,c,d,e,q^{\phi_2-K}
 \\
    a^{\frac{1}{2}},-a^{\frac{1}{2}},aq/b,aq/c,aq/d,aq/e,aq^{K-\phi_2+1} \end{array}
    ;q,\frac{a^2q^{K-\phi_2+2}}{bcde}\right],\label{n-ni1}
\end{gather}
where $a=t^3q^{m_1+m_2-1+\phi_2}$, $b=tq^{m_1}$, $c=t^2q^{K+m_1+m_2+\phi_2}$,
    $d=t$,
$e=t^3q^{m_2-1+2\phi_2}$, $\frac{a^2q^{K-\phi_2+2}}{bcde}=q/t.$
Applying Watson's transformation formula~(\ref{Watson}) to the~${}_8\phi_7$ series of~(\ref{n-ni1}), we have
\begin{gather}
 \mbox{${}_8\phi_7$ series of (\ref{n-ni1})}
=  \frac{(t^3q^{m_1+m_2+2\phi_2}, t^{-1}q^{m_1+1};q)_{K-\phi_2}}
    {(t^2q^{m_1+m_2+2\phi_2}, q^{m_1+1};q)_{K-\phi_2}}\nonumber\\
\hphantom{\mbox{${}_8\phi_7$ series of (\ref{n-ni1})}  =}{}\times
    {}_4\phi_3\left[\begin{array}{@{}c@{}} q^{\phi_2-K-m_2},t,t^3q^{m_2-1+2\phi_2},q^{\phi_2-K}
    \\ t^2q^{m_2+2\phi_2},tq^{\phi_2-K},
    tq^{\phi_2-K-m_1} \end{array};q,q\right].\label{n-ni2}
\end{gather}
Using Sears' transformation formulas \cite[p.~242, Appendix~III, equations~(III.15) and (III.16)]{GR},
we rewrite  the ${}_4\phi_3$ series in~(\ref{n-ni2}) as follows
\begin{gather}
     {}_4\phi_3\left[\begin{array}{@{}c@{}} q^{\phi_2-K-m_2},t,
    t^3q^{m_2-1+2\phi_2},q^{\phi_2-K}
    \\ t^2q^{m_2+2\phi_2},tq^{\phi_2-K},
    tq^{\phi_2-K-m_1} \end{array};q,q\right]
= \frac{(q^{m_1-K+\phi_2}, t^{-1}q^{1-m_2-K-\phi_2};q)_{K-\phi_2}}
    {(tq^{-m_1-K-\phi_2}, t^2q^{1-m_2-K-\phi_2};q)_{K-\phi_2}}\nonumber\\
\qquad{}\times
    {}_4\phi_3\left[\begin{array}{@{}c@{}} tq^{m_1},t,t^{-2}q^{1-m_2-K-\phi_2},q^{\phi_2-K}
    \\ tq^{\phi_2-K},t^{-1}q^{1-m_2-K-\phi_2},q^{m_1+1}
    \end{array};q,q\right]\label{n-ni3}.
\end{gather}
Sears' transformation gives us
\begin{gather}
 {}_4\phi_3 \text{ series in (\ref{n-ni3})}
 =
\frac{(t^{-2}q^{1-m_2-K-\phi_2}, q;q)_{K-\phi_2}}
    {(t^{-1}q^{1-m_2-K-\phi_2}, t^{-1}q;q)_{K-\phi_2}}\nonumber\\
\hphantom{{}_4\phi_3 \text{ series in (\ref{n-ni3})}  =}{}\times
    {}_4\phi_3\left[\begin{array}{@{}c@{}} t,t^2q^{m_1+m_2+k+\phi_2},t^{-1}q,q^{\phi_2-K}
    \\ q^{m_1+1},q,t^2q^{m_2} \end{array};q,q\right],\label{n-ni4}
\end{gather}
and
\begin{gather*}
{}_4\phi_3 \text{ series in (\ref{n-ni4})}
=
\frac{(t^{-1}q, tq^{m_2+2\phi_2};q)_{K-\phi_2}}
    {(q, t^2q^{m_2+2\phi_2};q)_{K-\phi_2}}t^{K-\phi_2}\nonumber\\
\hphantom{{}_4\phi_3 \text{ series in (\ref{n-ni4})}  =}{}\times
    {}_4\phi_3\left[\begin{array}{@{}c@{}} t,t^{-2}q^{1-m_2-K-\phi_2},tq^{m_1},q^{\phi_2-K}
    \\ q^{m_1+1},tq^{\phi_2-K},t^{-1}q^{1-m_2-K-\phi_2}
    \end{array};q,q\right].
\end{gather*}
Then we have
\begin{gather*}
 \mbox{l.h.s.\ of (\ref{n-ni1})} \\
{}= \sum_{\phi_2=0}^{K}\frac{(t;q)_{m_1}(t)_{m_2}
    (t^2q^{m_2+\phi_2};q)_{\phi_2}(t^2q^{-1};q)_{\phi_2}
    (t^{-1}q;q)_{K-\phi_2}(tq^{m_2+2\phi_2};q)_{K-\phi_2}}
    {(q;q)_{m_1}(q;q)_{m_2}(q;q)_{\phi_2}(q^{m_2+1};q)_{\phi_2}
    (t^2q^{m_2-1+\phi_2};q)_{\phi_2}(q;q)_{K-\phi_2}(t^2q^{m_2+2\phi_2};q)_{K-\phi_2}
}
\\
\quad {}\times {}_4\phi_3\left[\begin{array}{@{}c@{}} t,t^{-2}q^{1-m_2-K-\phi_2},tq^{m_1},q^{\phi_2-K}  \\
q^{m_1+1},tq^{\phi_2-K},t^{-1}q^{1-m_2-K-\phi_2} \end{array};q,q\right]  \\
{}= \sum_{\phi_1=0}^{K}\frac{ (t;q)_{m_1}(t;q)_{m_2}(t^{-1}q;q)_{K}(tq^{m_2};q)_{K}
    (q^{-K}, t^{-2}q^{1-m_2-K}, t, tq^{m_1};q)_{\phi_1} }
    {(q;q)_{m_1}(q;q)_{m_2}(q;q)_{K}(t^2q^{m_2};q)_{K}
    (tq^{-K}, t^{-1}q^{1-m_2-K}, q, q^{m_1+1};q)_{\phi_1}
}
\\
\quad {} \times \sum_{\phi_2=0}^{K-\phi_1}
    \frac{(t^2q^{m_2-1}, tq^{\frac{m_2+1}{2}}, -tq^{\frac{m_2+1}{2}},
t^2q^{-1}, tq^{m_2+K-\phi_1}, q^{\phi_1-K};q)_{\phi_2}}
    {(q, tq^{\frac{m_2-1}{2}}, -tq^{\frac{m_2-1}{2}}, q^{m_2+1}, tq^{\phi_1-K},
 t^2q^{m_2+K-\phi_1};q)_{\phi_2}}
    (t^2/q)^K(q/t)^{\phi_2}q^{\phi_1}  \\
{} =\sum_{\phi_1=0}^{K}\frac{(t;q)_{m_1}(t;q)_{m_2}(t^{-1}q;q)_{K}(tq^{m_2};q)_{K}
    (q^{-K}, t^{-2}q^{1-m_2-K}, t, tq^{m_1};q)_{\phi_1}}
    {(q;q)_{m_1}(q;q)_{m_2}(q;q)_{K}(t^2q^{m_2};q)_{K}
    (tq^{-K}, t^{-1}q^{1-m_2-K}, q, q^{m_1+1};q)_{\phi_1}}
     (t^2/q)^Kq^{\phi_1} \nonumber \\
\quad{} \times {}_6\phi_5\left[\begin{array}{@{}c@{}}
t^2q^{m_2-1},tq^{\frac{m_2+1}{2}},-tq^{\frac{m_2+1}{2}},t^2q^{-1},tq^{m_2+K-\phi_1},q^{\phi_1-K}
 \\
q,tq^{\frac{m_2-1}{2}},-tq^{\frac{m_2-1}{2}},q^{m_2+1},tq^{\phi_1-K},t^2q^{m_2+K-\phi_1} \end{array}
;q,q/t\right]  \\
{} = \sum_{\phi_1=0}^k\frac{(t;q)_{m_1}(t;q)_{m_2}(t^{-1}q;q)_{K}(tq^{m_2};q)_{K}
    (q^{-K}, t^{-2}q^{1-m_2-K}, t, tq^{m_1};q)_{\phi_1}}
    {(q;q)_{m_1}(q;q)_{m_2}(q;q)_{K}(t^2q^{m_2};q)_{K}
    (tq^{-K}, t^{-1}q^{1-m_2-K}, q, q^{m_1+1};q)_{\phi_1}}
 \\
\quad {} \times \frac{(t^2q^{m_2};q)_{K-\phi_1}(t^{-1}q^{1-K+\phi_1};q)_{K-\phi_1}}
    {(q^{m_2+1};q)_{K-\phi_1}(tq^{\phi_1-K};q)_{K-\phi_1}}(t^2/q)^Kq^{\phi_1}
    =\text{r.h.s.\ of (\ref{N=ippan_C})}.
\end{gather*}

\subsection{The general case}\label{general_C}

We have
\begin{gather}
\text{l.h.s.\ of (\ref{N=ippan_C})} =
\sum_{\phi_{2,n} = 0}^K
\prod_{2 \leq k \leq n}{(t;q)_{\phi_k} (t;q)_{\phi_k + m_k}
\over (q;q)_{\phi_k} (q;q)_{\phi_k +m_k}} \nonumber \\
\qquad {}\times
\prod_{2\leq l\leq n}
\frac{
(t^{n-l+1} q^{\phi_l + \phi_{l+1,n} + m_{l,n}};q)_
{\phi_l}
(t^{n-l+2}q^{ 2\phi_{l+1,n} + m_{l+1,n} -1};q)_
{\phi_l}}
{(t^{n-l+2}q^{\phi_l + \phi_{l+1,n} + m_{l,n} -1};q)_
{\phi_l}
(t^{n-l+1}q^{ 2\phi_{l+1,n} + m_{l+1,n}};q)_
{\phi_l}} \nonumber \\
\qquad {} \times
\sum_{\phi_1 = 0}^{K-\phi_{2,n}}
\frac{
(t^{n} q^{\phi_1 + \phi_{2,n} + m_{1,n}};q)_
{\phi_1}
(t^{n+1}q^{ 2\phi_{2,n} + m_{2,n} -1};q)_
{\phi_1}}
{(t^{n+1}q^{\phi_1 + \phi_{2,n} + m_{1,n} -1};q)_
{\phi_1}
(t^{n}q^{ 2\phi_{2,n} + m_{2,n}};q)_
{\phi_1}} \nonumber \\
\qquad{}\times
(t^2/q)^{K-\phi_{2,n} - \phi_1}
\frac{(t^{-1}q;q)_{K-\phi_{2,n} - \phi_1}(t^nq^{2\phi_{2,n} + m_{1,n} +2\phi_1};q)_{K-\phi_{2,n} - \phi_1}}
{(q;q)_{K-\phi_{2,n} - \phi_1} (t^{n+1}q^{2\phi_{2,n} + m_{1,n} +2\phi_1};q)_{K-\phi_{2,n} - \phi_1}}.\label{sono7_C}
\end{gather}
Here, we can describe the summation with respect to  $\phi_1$ in (\ref{sono7_C})  as follows:
\begin{gather}
\frac{(t;q)_{m_1}(t^{n}q^{m_{1,n}+2\phi_{2,n}}, t^{-1}q;q)_{K-\phi_{2,n}}
}
{(q;q)_{m_1}(t^{n+1}q^{m_{1,n}+2\phi_{2,n}}, q;q)_{K-\phi_{2,n}}
}(t^2/q)^{K-\phi_{2,n}} \nonumber\\
\qquad\quad{}\times
{}_8\phi_7 \biggl[
\begin{array}{@{}c@{}} t^{n+1}q^{m_{1,n}+2\phi_{2,n}-1},t^{\frac{n+1}{2}}
q^{\frac{m_{1,n}+2\phi_{2,n}+1}{2}},  \\
t^{\frac{n+1}{2}}q^{\frac{m_{1,n}+2\phi_{2,n}-1}{2}},-t^{\frac{n+1}{2}}
q^{\frac{m_{1,n}+2\phi_{2,n}-1}{2}},
\end{array} \nonumber \\
 \qquad\qquad{} \begin{array}{@{}c@{}}
 -t^{\frac{n+1}{2}}q^{\frac{m_{1,n}+2\phi_{2,n}+1}{2}},
t,tq^{m_1},t^{n+1}q^{m_{2,n}+2\phi_{2,n}-1},t^{n}q^{m_{1,n}+2\phi_{2,n}+K},q^{\phi_{2,n}-K} \\
t^{n}q^{m_{1,n}+2\phi_{2,n}},t^{n}
q^{m_{2,n}+2\phi_{2,n}},q^{m_1+1},tq^{\phi_{2,n}-K},t^{n+1}q^{m_{1,n}+2\phi_{2,n}+K}
\end{array};q,q/t\biggr] \nonumber \\
\qquad{} = \frac{ (t;q)_{m_1}(t^{n}q^{m_{1,n}+2\phi_{2,n}}, t^{-1}q,
t^{-n}q^{1-K-m_{2,n}-\phi_{2,n}};q)_{K-\phi_{2,n}} }
{ (q;q)_{m_1}(q^{m_1+1}, q, tq^{\phi_{2,n}-K};q)_{K-\phi_{2,n}} }
(t^2/q)^{K-\phi_{2,n}} \nonumber \\
 \qquad\quad{} \times  {}_4\phi_3\left[\begin{array}{@{}c@{}} t^{n-1}q^{m_{2,n}+2\phi_{2,n}},t^{n+1}
q^{m_{2,n}+2\phi_{2,n}-1},t^{n}q^{m_{1,n}+\phi_{2,n}+K},q^{\phi_{2,n}-K} \\
t^{n}q^{m_{1,n}+2\phi_{2,n}},t^{n}q^{m_{2,n}+2\phi_{2,n}},t^{n}q^{m_{2,n}+2\phi_{2,n}}
\end{array};
    q,q\right].
\label{n-L2}
\end{gather}
Applying Sears' ${}_4\phi_3$ transformation formula \cite[p.~41, equation~(2.10.4)]{GR} to the r.h.s.\  of
(\ref{n-L2}), we have
\begin{gather}
 \text{r.h.s.\ of (\ref{n-L2})}
 = (t^2/q)^{K-\phi_{2,n}}
\frac{ (t;q)_{m_1}(t^{-1}q;q)_{K-\phi_{2,n}}(t^{n-1}q^{m_{2,n} + 2\phi_{2,n}};q)_{K-\phi_{2,n}}
 }
{ (q;q)_{m_1}(q;q)_{K-\phi_{2,n}}(t^{n}q^{m_{2,n} + 2\phi_{2,n}};q)_{K-\phi_{2,n}} } \nonumber \\
\hphantom{\text{r.h.s.\ of (\ref{n-L2})} =}{}
\times
{}_4\phi_3\left[\begin{array}{@{}c@{}} tq^{m_1},t,t^{-n}q^{1-K-m_{2,n} -\phi_{2,n}},q^{\phi_{2,n}-K} \\
q^{m_1+1},tq^{\phi_{2,n}-K},t^{-n+1}q^{1-K-m_{2,n} -\phi_{2,n}} \end{array};q,q\right].
\label{n-L4}
\end{gather}
Using the following two formulas
\begin{gather*}
 \frac{(t^{-1}q;q)_{K-\phi_{2,n}}(q^{\phi_{2,n}-K};q)_{\phi_1}}
    {(q;q)_{K-\phi_{2,n}}(tq^{\phi_{2,n}-K};q)_{\phi_1}}
    =\frac{(t^{-1}q;q)_{K-\phi_{2,n}-\phi_1}}
    {(q;q)_{K-\phi_{2,n}-\phi_1}}
    t^{-\phi_1}, \\
\frac{(t^{n-1}q^{m_{2,n}+2\phi_{2,n}};q)_{K-\phi_{2,n}}
(t^{-n}q^{1-K-m_{2,n}-\phi_{2,n}};q)_{\phi_1}}
    {(t^{n}q^{m_{2,n}+2\phi_{2,n}};q)_{K-\phi_{2,n}}
(t^{-n+1}q^{1-K-m_{2,n}-\phi_{2,n}};q)_{\phi_1}}
    =\frac{(t^{n-1}q^{m_{2,n}+2\phi_{2,n}};q)_{K-\phi_{2,n}-\phi_1}}
    {(t^{n}q^{m_{2,n}+2\phi_{2,n}};q)_{K-\phi_{2,n}-\phi_1}}
    t^{-\phi_1},
\end{gather*}
we have
\begin{gather*}
 \mbox{r.h.s.\  of (\ref{n-L4})} \\
 =\sum_{\phi_1=0}^{K-\phi_{2,n}}
\frac{(t;q)_{m_1}(t;q)_{\phi_1}(tq^{m_1};q)_{\phi_1}
(t^{-1}q;q)_{K-\phi_{2,n}-\phi_1}(t^{n-1}q^{m_{2,n} + 2\phi_{2,n}};q)_{K-\phi_{2,n}-\phi_1}}
    {(q;q)_{m_1}(q;q)_{\phi_1}(q^{m_1+1};q)_{\phi_1}
    (q;q)_{K-\phi_{2,n}-\phi_1}(t^{n}q^{m_{2,n} + 2\phi_{2,n}};q)_{K-\phi_{2,n}-\phi_1}} \\
\quad{}\times
    (t^2/q)^{K-\phi_{2,n}-\phi_1}.
\end{gather*}
Then we have
\begin{gather}
\text{l.h.s.\  of } (\ref{sono7_C})
=
\sum_{\phi_{2,n} = 0}^K
\prod_{2 \leq k \leq n}{(t;q)_{\phi_k} (t;q)_{\phi_k + m_k}
\over (q;q)_{\phi_k} (q;q)_{\phi_k +m_k}} \nonumber \\
\quad{}\times
\prod_{2\leq l\leq n}
\frac{
(t^{n-l+1} q^{\phi_l + \phi_{l+1,n} + m_{l,n}};q)_
{\phi_l}
(t^{n-l+2}q^{ 2\phi_{l+1,n} + m_{l+1,n} -1};q)_
{\phi_l}}
{(t^{n-l+2}q^{\phi_l + \phi_{l+1,n} + m_{l,n} -1};q)_
{\phi_l}
(t^{n-l+1}q^{ 2\phi_{l+1,n} + m_{l+1,n}};q)_
{\phi_l}} \nonumber \\
\quad{}\times\sum_{\phi_1=0}^{K-\phi_{2,n}}
\frac{(t;q)_{m_1}(t;q)_{\phi_1}(tq^{m_1};q)_{\phi_1}
(t^{-1}q;q)_{K-\phi_{2,n}-\phi_1}(t^{n-1}q^{m_{2,n} + 2\phi_{2,n}};q)_{K-\phi_{2,n}-\phi_1}}
    {(q;q)_{m_1}(q;q)_{\phi_1}(q^{m_1+1};q)_{\phi_1}
    (q;q)_{K-\phi_{2,n}-\phi_1}(t^{n}q^{m_{2,n} + 2\phi_{2,n}};q)_{K-\phi_{2,n}-\phi_1}} \nonumber\\
    \qquad \quad {}\times
    (t^2/q)^{K-\phi_{2,n}-\phi_1} \nonumber \\
{}=
\sum_{\phi_1 = 0}^K
{(t;q)_{\phi_1} (t;q)_{\phi_1 + m_1}
\over (q;q)_{\phi_1} (q;q)_{\phi_1 +m_1}}
\sum_{\phi_{2,n} = 0}^{K-\phi_1}
\prod_{2 \leq k \leq n}{(t;q)_{\phi_k} (t;q)_{\phi_k + m_k}
\over (q;q)_{\phi_k} (q;q)_{\phi_k +m_k}} \nonumber \\
\quad{}\times
\prod_{2\leq l\leq n}
\frac{
(t^{n-l+1} q^{\phi_l + \phi_{l+1,n} + m_{l,n}};q)_
{\phi_l}
(t^{n-l+2}q^{ 2\phi_{l+1,n} + m_{l+1,n} -1};q)_
{\phi_l}}
{(t^{n-l+2}q^{\phi_l + \phi_{l+1,n} + m_{l,n} -1};q)_
{\phi_l}
(t^{n-l+1}q^{ 2\phi_{l+1,n} + m_{l+1,n}};q)_
{\phi_l}} \nonumber \\
\quad{}\times
\frac{
(t^{-1}q;q)_{K-\phi_1-\phi_{2,n}}(t^{n-1}q^{m_{2,n} + 2\phi_{2,n}};q)_{K-\phi_1-\phi_{2,n}}}
    {    (q;q)_{K-\phi_1-\phi_{2,n}}(t^{n}q^{m_{2,n} + 2\phi_{2,n}};q)_{K-\phi_1-\phi_{2,n}}}
    (t^2/q)^{K-\phi_1-\phi_{2,n}} \nonumber \\
{}=
\sum_{\phi_1 = 0}^K
{(t;q)_{\phi_1} (t;q)_{\phi_1 + m_1}
\over (q;q)_{\phi_1} (q;q)_{\phi_1 +m_1}}
\sum_{\phi_{2}+\cdots+\phi_{n}+i = K-\phi_1}
\prod_{2 \leq k \leq n}{(t;q)_{\phi_k} (t;q)_{\phi_k + m_k}
\over (q;q)_{\phi_k} (q;q)_{\phi_k +m_k}} \nonumber \\
\quad{}\times
\frac{(t^{-1}q;q)_{i}(t^{n-1}q^{m_{2,n} + 2\phi_{2,n}};q)_{i}}
    {    (q;q)_{i}(t^{n}q^{m_{2,n} + 2\phi_{2,n}};q)_{i}}
    (t^2/q)^{i} \nonumber \\
\quad{}\times
\prod_{2\leq l\leq n}
\frac{
(t^{n-l+1} q^{\phi_l + \phi_{l+1,n} + m_{l,n}};q)_
{\phi_l}
(t^{n-l+2}q^{ 2\phi_{l+1,n} + m_{l+1,n} -1};q)_
{\phi_l}}
{(t^{n-l+2}q^{\phi_l + \phi_{l+1,n} + m_{l,n} -1};q)_
{\phi_l}
(t^{n-l+1}q^{ 2\phi_{l+1,n} + m_{l+1,n}};q)_
{\phi_l}}.\label{n-L5}
\end{gather}
By the induction hypothesis, we obtain
\begin{gather*}
 \mbox{r.h.s.\ of (\ref{n-L5})}
 =\sum_{\phi_1 = 0}^K
{(t;q)_{\phi_1} (t;q)_{\phi_1 + m_1}
\over (q;q)_{\phi_1} (q;q)_{\phi_1 +m_1}}
\sum_{\phi_2+\cdots+\phi_n=K-\phi_1}
\prod_{2 \leq k \leq n}{(t;q)_{\phi_k} (t;q)_{\phi_k + m_k}
\over (q;q)_{\phi_k} (q;q)_{\phi_k +m_k}} \\
\hphantom{\mbox{r.h.s.\ of (\ref{n-L5})} }{} =
\sum_{\phi_1+\phi_2+\cdots+\phi_n=K}
\prod_{1 \leq j \leq n}{(t;q)_{\phi_j} (t;q)_{\phi_j + m_j}
\over (q;q)_{\phi_j} (q;q)_{\phi_j +m_j}} = \text{r.h.s.\ of (\ref{N=ippan_C})}.
\end{gather*}

\section[Tableau formulas for Macdonald polynomials of type $C_n$]{Tableau formulas for Macdonald polynomials of type $\boldsymbol{C_n}$}\label{section6}

In this section, we establish the  tableau formulas
for the Macdonald polynomials
of type~$C_n$.
Let $I := \{ 1,2,\dots,n-1, n, \overline{n}, \overline{n-1}, \dots, \overline{1} \}$
be the index set with the ordering
\begin{gather}\label{OrderOfC}
1 \prec 2 \prec \cdots \prec {n-1}\prec  {n} \prec \overline{n} \prec \overline{n-1}
\prec \cdots \prec \overline{1}.
\end{gather}
Denoting by $\Lambda_1$ the f\/irst fundamental weight of type $C_n$, let $P_{(r)}^{(C_n)}(x;q,t,T)$ be the Macdonald
polynomials of type $C_n$ associated with the weights~$r \Lambda_1$ for
$r \in {\mathbb{Z}_{\geq 0}}$.

The following theorem \cite[Theorem 5.1]{HNS} was conjectured by
Lassalle~\cite{La}.
\begin{thm}[\cite{HNS, La}]\label{Lassalle_C} For any positive integer $r$ we have
\begin{gather}
 G_r(x;q,t) = \sum^{[r/2]}_{i = 0}
\frac{(t;q)_{r-2i}}{(q;q)_{r-2i}}P_{(r-2i)}^{(C_n)}(x;q,t,T)T^i
\frac{(t/T;q)_i(t^n q^{r-2i};q )_i}{(q;q)_i(T t^{n-1} q^{r-2i +1};q)_i}.\label{G_to_P_C}
\end{gather}
Conversely
\begin{gather*}
 P_{(r)}^{(C_n)}(x;q,t,T) = {(q;q)_r \over (t;q)_r}
\sum^{[r/2]}_{i = 0}
G_{r-2i}(x;q,t)t^i
\frac{(T/t;q)_i(t^n q^{r-i};q )_i}{(q;q)_i(T t^{n-1} q^{r-i};q)_i}{1-t^n q^{r-2i} \over 1-t^n q^{r-i}}.
\end{gather*}
\end{thm}

First we prove the tableau formula for $P^{(C_n)}_{(r)}(x;q,t,t^2/q)$.
\begin{thm}\label{main_C} We have
\begin{gather}
P_{(r)}^{(C_n)}\big(x;q,t,t^2/q\big)=
\frac{(q;q)_r}{(t;q)_r}
    \displaystyle\sum_{\theta_{1}+\theta_{2}+\cdots+
    \theta_{\overline{1}}=r}
\prod_{k \in I}
\frac{(t;q)_{\theta_{k}}}{(q;q)_{\theta_{k}}}
\label{TSF_C}
\\
\qquad {}\times \prod_{1\leq l\leq n}\frac{(t^{n-l+1}q^{\theta_{l}+
    \cdots+\theta_{\overline{l+1}}};q)_{\theta_{\overline{l}}}
    (t^{n-l+2}q^{\theta_{l+1}+\cdots+\theta_{\overline{l+1}}-1};q)_
    {\theta_{\overline{l}}}}
    {(t^{n-l+2}q^{\theta_{l}+\cdots+\theta_
    {\overline{l+1}}-1};q)_{\theta_{\overline{l}}}
    (t^{n-l+1}q^{\theta_{l+1}+\cdots+\theta_{\overline{l+1}}};q)_
    {\theta_{\overline{l}}}}
    x_1^{\theta_{1}-\theta_{\overline{1}}} x_2^{\theta_{2}-\theta_{\overline{2}}}\cdots
    x_n^{\theta_{n}-\theta_{\overline{n}}}.\nonumber
\end{gather}
\end{thm}

\begin{rem}
It would be an intriguing problem to show the factorization of the Macdonald polynomial $P_{(r)}^{(C_n)}(x;q,t,t^2/q)$
from our formula~(\ref{TSF_C})
when we  make the principal specialization
\begin{gather*}
P_{(r)}^{(C_n)}\big(\big(t^2/q\big)^{1/2} t^{n-1},\ldots,\big(t^2/q\big)^{1/2}t,\big(t^2/q\big)^{1/2};q,t,t^2/q\big)
=q^{r/2}t^{-r n}{(t^n;q)_r (t^{2(n+1)}/q^2;q)_r \over (t;q)_r (t^{n+1}/q;q)_r}.
\end{gather*}
\end{rem}

\begin{rem}
Setting $n=1$ in~(\ref{TSF_C}) we have,
\begin{gather*}
P_{(r)}^{(C_1)}\big(x;q,t,t^2/q\big) =
\sum_{\theta_1 = 0}^r\!
\frac{(t^2/q, q^{-r};q)_{\theta_1}}
{(q, t^{-2}q^{2-r};q)_{\theta_1}}
(q/t)^{2 \theta_1} x^{-r +2 \theta_1}=
x^{-r} \,{}_2\phi_1\!\left[ {t^2/q, q^{-r} \atop  t^{-2}q^{2-r}}\!;q,(q x/t)^2  \right].
\end{gather*}
Setting $n=2$ in~(\ref{TSF_C}) we have
\begin{gather*}
P_{(r)}^{(C_2)}\big(x;q,t,t^2/q\big) =
\frac{(q;q)_r}{(t;q)_r}
\sum_{\theta_1 + \theta_2 + \theta_{\overline{2}} + \theta_{\overline{1}} = r}
\frac{(t;q)_{\theta_1} (t;q)_{\theta_{\overline{1}}} (t;q)_{\theta_2} (t;q)_{\theta_{\overline{2}}} }
{(q;q)_{\theta_1} (q;q)_{\theta_{\overline{1}}} (q;q)_{\theta_2} (q;q)_{\theta_{\overline{2}}} }\nonumber \\
\hphantom{P_{(r)}^{(C_2)}\big(x;q,t,t^2/q\big) =}{}
 \times
\frac{(t^2 q^{\theta_1 + \theta_2 + \theta_{\overline{2}}},
t^3 q^{\theta_2 + \theta_{\overline{2}}-1};q)_{\theta_{\overline{1}}}}
{(t^3 q^{\theta_1 + \theta_2 + \theta_{\overline{2}}-1},
t^2 q^{\theta_2 + \theta_{\overline{2}}};q)_{\theta_{\overline{1}}}}
\frac{(t q^{\theta_2},
t^2/q;q)_{\theta_{\overline{2}}}}
{(t^2 q^{\theta_2 -1},
t;q)_{\theta_{\overline{2}}}}
 x_1^{\theta_1 - \theta_{\overline{1}}} x_2^{\theta_2 -\theta_{\overline{2}}}.
\end{gather*}
\end{rem}

\begin{proof}[Proof of Theorem~\ref{main_C}]
Let $\big\{\Psi^{(C_n)}_{(r)}(x;q,t,T)\big\}_{r\in \mathbb{Z}_{\geq 0}}$ be a certain collection of  Laurent polynomials.
By using Lassalle's formula~(\ref{G_to_P_C}),
it is proved by induction that the inf\/inite system of equalities for $\Psi^{(C_n)}_{(r)}(x;q,t,T)$
\begin{gather}
G_r(x;q,t) = \sum^{[r/2]}_{i = 0}
\frac{(t;q)_{r-2i}}{(q;q)_{r-2i}}\Psi_{(r-2i)}^{(C_n)}(x;q,t,T)T^i
\frac{(t/T;q)_i(t^n q^{r-2i};q )_i}{(q;q)_i(T t^{n-1} q^{r-2i +1};q)_i}\label{G_to_Psi_C}
\end{gather}
gives us $\Psi_{(r)}^{(C_n)}(x;q,t,T)=P_{(r)}^{(C_n)}(x;q,t,T)$, $r\in \mathbb{Z}_{\geq 0}$.
We use this argument with the specialization of the parameter $T = t^2/q$.

Set
\begin{gather*}
 \Psi_{(r)}^{(C_n)}\big(x;q,t,t^2/q\big)=
\frac{(q;q)_r}{(t;q)_r}
    \displaystyle\sum_{\theta_{1}+\theta_{2}+\cdots+
    \theta_{\overline{1}}=r}
\prod_{k \in I}
\frac{(t;q)_{\theta_{k}}}{(q;q)_{\theta_{k}}}
\\
\qquad{} \times \prod_{1\leq l\leq n}\frac{(t^{n-l+1}q^{\theta_{l}+
    \cdots+\theta_{\overline{l+1}}};q)_{\theta_{\overline{l}}}
    (t^{n-l+2}q^{\theta_{l+1}+\cdots+\theta_{\overline{l+1}}-1};q)_
    {\theta_{\overline{l}}}}
    {(t^{n-l+2}q^{\theta_{l}+\cdots+\theta_
    {\overline{l+1}}-1};q)_{\theta_{\overline{l}}}
    (t^{n-l+1}q^{\theta_{l+1}+\cdots+\theta_{\overline{l+1}}};q)_
    {\theta_{\overline{l}}}}
    x_1^{\theta_{1}-\theta_{\overline{1}}} x_2^{\theta_{2}-\theta_{\overline{2}}}\cdots
    x_n^{\theta_{n}-\theta_{\overline{n}}}.
\end{gather*}
We prove this family of  Laurent polynomials satisf\/ies~(\ref{G_to_Psi_C}) with the specialization $T = t^2/q$.
In view of the  $(\mathbb{Z}/2\mathbb{Z})^n$ symmetry of $\Psi^{(C_n)}_{(r)}(x;q,t,t^2/q)$,
it is suf\/f\/icient to consider in~(\ref{G_to_Psi_C}) the
coef\/f\/icients of the monomials $x_1^{m_1}\cdots x_n^{m_n}$
with nonnegative powers $m_1,\dots,m_n \in \mathbb{Z}_{\geq 0}$ only.
Let $r\in \mathbb{Z}_{\geq 0}$, and
f\/ix $m_1,\dots,m_n \in \mathbb{Z}_{\geq 0}$ arbitrarily. Set  $K := {1 \over 2}(r - m_1 - m_2 - \cdots - m_n)$ for simplicity.
Setting
\begin{gather*}
\theta_k  = m_k +\phi_k, \qquad \theta_{\overline{k}} =\phi_k,\qquad  1 \leq k \leq n,
\end{gather*}
one f\/inds that the
coef\/f\/icients of the monomials $x_1^{m_1}\cdots x_n^{m_n}$ in~(\ref{G_to_Psi_C})  is exactly given by l.h.s.\ of~(\ref{N=ippan_C}).
On the other hand, the coef\/f\/icients of the monomials $x_1^{m_1}\cdots x_n^{m_n}$ in~$G_r(x;q,t) $ is r.h.s.\ of~(\ref{N=ippan_C}).
Hence we have proved~(\ref{G_to_Psi_C}) with $T = t^2/q$,
which establishes the tableau formula $P^{(C_n)}_{(r)}(x;q,t,t^2/q)=\Psi^{(C_n)}_{(r)}(x;q,t,t^2/q)$.
\end{proof}

Finally, we present a tableau formula for the one-row Macdonald polynomials $P_{(r)}^{(C_n)}(x;q,t,T)$
with general parameters $(q,t,T)$.
\begin{thm}\label{main_C2}
Set $\theta := \min(\theta_{n}, \theta_{\overline{n}})$. We have
\begin{gather}
P_{(r)}^{(C_n)}(x;q,t,T)=
\frac{(q;q)_r}{(t;q)_r}
\sum_{\theta_{1}+\theta_{2}+\cdots+
    \theta_{\overline{1}}=r}
\prod_{k \in I {\setminus} \{ n, \overline{n} \} }
\frac{(t;q)_{\theta_{k}}}{(q;q)_{\theta_{k}}}
\frac{(t;q)_{|\theta_{n} - \theta_{\overline{n}}|}}{(q;q)_{|\theta_{n} - \theta_{\overline{n}}|}}
\notag
\\
\hphantom{P_{(r)}^{(C_n)}(x;q,t,T)=}{}\times
 \prod_{1\leq l\leq n-1} \Biggl(
\frac{
(t^{n-l-1}q^{\theta_{l}+
    \cdots+\theta_{n-1}+|\theta_{n} - \theta_{\overline{n}}|+\theta_{\overline{n-1}}+\cdots +
   \theta_{\overline{l+1}} +1};q)_{\theta_{\overline{l}}}
}
{
(t^{n-l}q^{\theta_{l}+\cdots+\theta_{n-1}+|\theta_{n} - \theta_{\overline{n}}|+\theta_{\overline{n-1}}
+\cdots+
\theta_
    {\overline{l+1}}};q)_{\theta_{\overline{l}}}
} \nonumber \\
\qquad\hphantom{P_{(r)}^{(C_n)}(x;q,t,T)=}{}\times
    \frac{
    (t^{n-l}q^{\theta_{l+1}+\cdots\theta_{n-1}+|\theta_{n} - \theta_{\overline{n}}|
   +\theta_{\overline{n-1}}+\cdots+\theta_{\overline{l+1}}};q)_
    {\theta_{\overline{l}}}}
{
    (t^{n-l-1}q^{\theta_{l+1}+\cdots+
\theta_{n-1}+|\theta_{n} - \theta_{\overline{n}}|+\theta_{\overline{n-1}}+\cdots
+\theta_{\overline{l+1}}+1};q)_
    {\theta_{\overline{l}}}} \Biggr) \nonumber \\
\hphantom{P_{(r)}^{(C_n)}(x;q,t,T)=}{}
\times
{(T;q)_{\theta}(t^n q^{r-2\theta};q)_{2\theta} \over (q;q)_{\theta}(T t^{n-1} q^{r-\theta};q )_{\theta}
(t^{n-1} q^{r-2\theta +1};q)_{\theta}}
    x_1^{\theta_{1}-\theta_{\overline{1}}} x_2^{\theta_{2}-\theta_{\overline{2}}}\cdots
    x_n^{\theta_{n}-\theta_{\overline{n}}}. \label{general-C}
\end{gather}
\end{thm}

\begin{rem}
It would be an intriguing problem to show the factorization of the Macdonald polynomial $P_{(r)}^{(C_n)}(x;q,t,T)$
from our formula~(\ref{general-C})
when we make the principal specialization:
\begin{gather*}
P_{(r)}^{(C_n)}\big(T^{1/2} t^{n-1},\ldots,T^{1/2}t,T^{1/2};q,t,T\big)
=T^{-r/2} t^{-r (n-1)}{(t^n;q)_r (t^{2(n-1)}T;q)_r \over (t;q)_r (t^{n}T;q)_r}.
\end{gather*}
\end{rem}

\begin{proof}[Proof of Theorem~\ref{main_C2}]
We prove that the system of equalities~(\ref{G_to_Psi_C}) is satisf\/ied by
the following Laurent polynomials
\begin{gather*}
 \Psi_{(r)}^{(C_n)}(x;q,t,T)=
\frac{(q;q)_r}{(t;q)_r}
    \displaystyle\sum_{\theta_{1}+\theta_{2}+\cdots+
    \theta_{\overline{1}}=r}
\prod_{k \in I {\setminus} \{ n, \overline{n} \} }
\frac{(t;q)_{\theta_{k}}}{(q;q)_{\theta_{k}}}
\frac{(t;q)_{|\theta_{n} - \theta_{\overline{n}}|}}{(q;q)_{|\theta_{n} - \theta_{\overline{n}}|}}
\\
\hphantom{\Psi_{(r)}^{(C_n)}(x;q,t,T)=}{}
\times \prod_{1\leq l\leq n-1}\Biggl(
\frac{
(t^{n-l-1}q^{\theta_{l}+
    \cdots+\theta_{n-1}+|\theta_{n} - \theta_{\overline{n}}|+\theta_{\overline{n-1}}+\cdots +
   \theta_{\overline{l+1}} +1};q)_{\theta_{\overline{l}}}
}
{
(t^{n-l}q^{\theta_{l}+\cdots+\theta_{n-1}+|\theta_{n} - \theta_{\overline{n}}|+\theta_{\overline{n-1}}
+\cdots+
\theta_
    {\overline{l+1}}};q)_{\theta_{\overline{l}}}
}  \\
\qquad\hphantom{\Psi_{(r)}^{(C_n)}(x;q,t,T)=}{}
\times
\frac{
    (t^{n-l}q^{\theta_{l+1}+\cdots\theta_{n-1}+|\theta_{n} - \theta_{\overline{n}}|
  +\theta_{\overline{n-1}}+\cdots+\theta_{\overline{l+1}}};q)_
    {\theta_{\overline{l}}}
}
{
    (t^{n-l-1}q^{\theta_{l+1}+\cdots+
\theta_{n-1}+|\theta_{n} - \theta_{\overline{n}}|+\theta_{\overline{n-1}}+\cdots
+\theta_{\overline{l+1}}+1};q)_
    {\theta_{\overline{l}}}}\Biggr) \\
\hphantom{\Psi_{(r)}^{(C_n)}(x;q,t,T)=}{}\times
{(T;q)_{\theta}(t^n q^{r-2\theta};q)_{2\theta} \over (q;q)_{\theta}(T t^{n-1} q^{r-\theta};q )_{\theta}
(t^{n-1} q^{r-2\theta +1};q)_{\theta}}
    x_1^{\theta_{1}-\theta_{\overline{1}}} x_2^{\theta_{2}-\theta_{\overline{2}}}\cdots
    x_n^{\theta_{n}-\theta_{\overline{n}}}.
\end{gather*}
Note that these $\Psi^{(C_n)}_{(r)}(x;q,t,T)$ have the $(\mathbb{Z}/2\mathbb{Z})^n$ symmetry.

Let $r\in \mathbb{Z}_{\geq 0}$, and
f\/ix $m_1,\ldots,m_n \in \mathbb{Z}_{\geq 0}$ arbitrarily.
We study  the coef\/f\/icient of the monomial $  x_1^{m_1} x_2^{m_2}\cdots
    x_n^{m_n}$ in  r.h.s.\ of (\ref{G_to_Psi_C}),
namely  we consider the case
$\theta_k - \theta_{\overline{k}} = m_k \geq 0$ $(1 \leq k \leq n)$.
Set $K := {1 \over 2}(r - m_1 - m_2 - \cdots - m_n)$,
$\theta_{l,m}:=\sum\limits_{k=l}^{m}\theta_{\overline{k}}$,  $ \check{I}:=I {\setminus} \{ n, \overline{n} \}$ for simplicity.
Then the coef\/f\/icient of the monomial $  x_1^{m_1} x_2^{m_2}\cdots
    x_n^{m_n}$ in  r.h.s.\ of~(\ref{G_to_Psi_C}) is expressed as follows
\begin{gather}
\sum_{\theta_1+\theta_2+\cdots+\theta_{\overline{1}}+2i=r}
\prod_{k \in \check{I}}
\frac{(t;q)_{\theta_{k}}}{(q;q)_{\theta_{k}}}
\frac{(t;q)_{m_n}}{(q;q)_{m_n}}
\notag
\\
\qquad\quad{} \times \prod_{1\leq l\leq n-1}\frac{
(t^{n-l-1}q^{\theta_{\overline{l}} + 2\theta_{l,n-1}
+m_{l,n} +1};q)_{\theta_{\overline{l}}}
    (t^{n-l}q^{2\theta_{l+1,n-1}+m_{l+1,n}};q)_
    {\theta_{\overline{l}}}}
{
(t^{n-l}q^{\theta_{\overline{l}}+2\theta_{l+1,n-1} + m_{l,n}};q)_{\theta_{\overline{l}}}
    (t^{n-l-1}q^{2\theta_{l+1,n-1} + m_{l+1,n}+1};q)_
    {\theta_{\overline{l}}}} \notag \\
\qquad\quad{}\times
{(T;q)_{\theta}(t^n q^{r-2\theta};q)_{2\theta} \over (q;q)_{\theta}(T t^{n-1} q^{r-\theta};q )_{\theta}
(t^{n-1} q^{r-2\theta +1};q)_{\theta}}
{(t/T;q)_i (t^n q^{r-2i};q)_i \over (q;q)_i (Tt^{n-1}q^{r-2i+1};q)_i}T^i \nonumber\\
\qquad{}=
\sum_{\theta_{1,n-1}=0}^{K} \sum_{\theta_{\overline{n}}=0}^{K-\theta_{1,n-1}}
\prod_{k \in \check{I}}
\frac{(t;q)_{\theta_{k}}}{(q;q)_{\theta_{k}}}
\frac{(t;q)_{m_n}}{(q;q)_{m_n}}
\notag
\\
\qquad\quad{} \times \prod_{1\leq l\leq n-1}\frac{
(t^{n-l-1}q^{\theta_{\overline{l}} + 2\theta_{l,n-1}
+m_{l,n} +1};q)_{\theta_{\overline{l}}}
    (t^{n-l}q^{2\theta_{l+1,n-1}+m_{l+1,n}};q)_
    {\theta_{\overline{l}}}}
{
(t^{n-l}q^{\theta_{\overline{l}}+2\theta_{l+1,n-1} + m_{l,n}};q)_{\theta_{\overline{l}}}
    (t^{n-l-1}q^{2\theta_{l+1,n-1} + m_{l+1,n}+1};q)_
    {\theta_{\overline{l}}}} \notag \\
\qquad\quad{}\times
{(T;q)_{\theta_{\overline{n}}}(t^n q^{2\theta_{1,n-1} + m_{1,n}};q)_{2\theta_{\overline{n}}}
\over
 (q;q)_{\theta_{\overline{n}}}(T t^{n-1} q^{2\theta_{1,n-1} + m_{1,n} +\theta_{\overline{n}}};q )_{\theta_{\overline{n}}}
(t^{n-1} q^{2\theta_{1,n-1} +m_{1,n} +1};q)_{\theta_{\overline{n}}}} \notag \\
\qquad\quad{}\times
{(t/T;q)_{K-\theta_{1,n-1} - \theta_{\overline{n}}} (t^n q^{2\theta_{1,n-1} + m_{1,n} +
2\theta_{\overline{n}}};q)_{K-\theta_{1,n-1} - \theta_{\overline{n}}}
\over
(q;q)_{K-\theta_{1,n-1} - \theta_{\overline{n}}}
(Tt^{n-1}q^{2\theta_{1,n-1} + m_{1,n} + 2\theta_{\overline{n}} +1};q)_{K-\theta_{1,n-1} - \theta_{\overline{n}}}}
T^{K-\theta_{1,n-1} - \theta_{\overline{n}}}\nonumber\\
\qquad {}=
\sum_{\theta_{1,n-1}=0}^{K}
\prod_{k \in \check{I}}
\frac{(t;q)_{\theta_{k}}}{(q;q)_{\theta_{k}}}
\frac{(t;q)_{m_n}}{(q;q)_{m_n}}
\notag
\\
\qquad\quad{} \times \prod_{1\leq l\leq n-1}\frac{
(t^{n-l-1}q^{\theta_{\overline{l}} + 2\theta_{l,n-1}
+m_{l,n} +1};q)_{\theta_{\overline{l}}}
    (t^{n-l}q^{2\theta_{l+1,n-1}+m_{l+1,n}};q)_
    {\theta_{\overline{l}}}}
{
(t^{n-l}q^{\theta_{\overline{l}}+2\theta_{l+1,n-1} + m_{l,n}};q)_{\theta_{\overline{l}}}
    (t^{n-l-1}q^{2\theta_{l+1,n-1} + m_{l+1,n}+1};q)_
    {\theta_{\overline{l}}}} \notag \\
\qquad\quad{}\times
{(t/T;q)_{K - \theta_{1,n-1}}
(t^n q^{2\theta_{1,n-1} + m_{1,n}};q)_{K - \theta_{1,n-1}}
\over
 (q;q)_{K - \theta_{1,n-1}}
(T t^{n-1} q^{2\theta_{1,n-1} + m_{1,n} +1};q )_{K - \theta_{1,n-1}}}
T^{K - \theta_{1,n-1}} \notag \\
\qquad\quad{}\times
{}_{6}W_{5}\big(Tt^{n-1}q^{2\theta_{1,n-1} +m_{1,n}}, T, t^n q^{\theta_{1,n-1} +m_{1,n}+K},
q^{-K+\theta_{1,n-1}};q,q/t\big).
    \label{main_C2_2}
\end{gather}
By the summation formula for ${}_{6}\phi_{5}$ series
\cite[p.~34, equation~(2.4.2)]{GR}, we have
\begin{gather*}
{}_{6}W_{5} \text{ series in  (\ref{main_C2_2})}
=
{(T t^{n-1} q^{2\theta_{1,n-1} + m_{1,n} +1};q )_{K - \theta_{1,n-1}}
(t^{-1} q^{-K+\theta_{1,n} +1};q)_{K - \theta_{1,n-1}}
\over
(t^{n-1} q^{2\theta_{1,n-1} + m_{1,n}+1};q)_{K - \theta_{1,n-1}}
(T t^{-1} q^{-K+\theta_{1,n-1} -1}  ;q)_{K - \theta_{1,n-1}}}
.
\end{gather*}
Note that the dependence on the parameter $T$ in~(\ref{main_C2_2}) disappears by the cancelation as
\begin{gather*}
{(t/T;q)_{K - \theta_{1,n-1}} (t^{-1} q^{-K + \theta_{1,n-1} +1};q)_{K - \theta_{1,n-1}}
\over
(T t^{-1} q^{-K+\theta_{1,n-1} -1}  ;q)_{K - \theta_{1,n-1}}
}
T^{K - \theta_{1,n-1}} = (t;q)_{K - \theta_{1,n-1}}.
\end{gather*}
Hence we have recast
the coef\/f\/icient of  the monomial $  x_1^{m_1} x_2^{m_2}\cdots
    x_n^{m_n}$ in  r.h.s.\ of~(\ref{G_to_Psi_C})
as
\begin{gather*}
\sum_{\theta_{1,n-1}=0}^{K}
\prod_{k \in \check{I}}
\frac{(t;q)_{\theta_{k}}}{(q;q)_{\theta_{k}}}
\frac{(t;q)_{m_n}}{(q;q)_{m_n}}
\notag
\\
\qquad{}\times \prod_{1\leq l\leq n-1}\frac{
(t^{n-l-1}q^{\theta_{\overline{l}} + 2\theta_{l,n-1}
+m_{l,n} +1};q)_{\theta_{\overline{l}}}
    (t^{n-l}q^{2\theta_{l+1,n-1}+m_{l+1,n}};q)_
    {\theta_{\overline{l}}}}
{
(t^{n-l}q^{\theta_{\overline{l}}+2\theta_{l+1,n-1} + m_{l,n}};q)_{\theta_{\overline{l}}}
    (t^{n-l-1}q^{2\theta_{l+1,n-1} + m_{l+1,n}+1};q)_
    {\theta_{\overline{l}}}} \notag \\
\qquad{}\times
{(t;q )_{K - \theta_{1,n-1}}
(t^{n} q^{2\theta_{1,n-1} + m_{1,n}};q)_{K - \theta_{1,n-1}}
\over
(q;q )_{K - \theta_{1,n-1}}
(t^{n-1} q^{2\theta_{1,n-1} + m_{1,n}+1};q)_{K - \theta_{1,n-1}}}
\end{gather*}
Changing the running indices as $\phi_k=\theta_{\overline{k}}$,
one f\/inds that this is
nothing but l.h.s.\  of~(\ref{sono6}). Then Theorem~\ref{N=ippan} means that this is the
coef\/f\/icient of the monomial $  x_1^{m_1} x_2^{m_2}\cdots
x_n^{m_n}$ in  $G_r(x;q,t)$.
\end{proof}

\section[Deformed $\mathcal{W}$ algebras of types $C_l$ and $D_l$]{Deformed $\boldsymbol{\mathcal{W}}$ algebras of types $\boldsymbol{C_l}$ and $\boldsymbol{D_l}$}\label{section7}

In this section, we study a relation between the tableau formulas for Macdonald polynomials
of types~$C_l$ and~$D_l$ and the deformed~$\mathcal{W}$ algebras of types~$C_l$ and~$D_l$.
We brief\/ly recall the def\/inition of the deformed~$\mathcal{W}$ algebras
of types~$C_l$ and~$D_l$~\cite{FR}.

Let $\{ \alpha_1, \alpha_2, \dots, \alpha_l \}$ and
$\{ \omega_1, \omega_2, \dots, \omega_l \}$ be the sets of simple roots and
of fundamental weights of a simple Lie algebra $\mathfrak{g}$ of rank~$l$.
Let $(\cdot, \cdot)$ be the invariant inner product on $\mathfrak{g}$ and
$C = (C_{i,j})_{1 \leq i, j \leq l}$  the {\it Cartan matrix} where
$C_{i, j} = 2(\alpha_i, \alpha_j) / (\alpha_i, \alpha_i)$.
Let $r^{\vee}$ be the maximal number of edges connecting two vertices of
the Dynkin diagram of $\mathfrak{g}$ and set $D=\operatorname{diag}(r_1, r_2, \dots, r_l)$
where $r_i = r^{\vee} (\alpha_i, \alpha_i) / 2$. Denote by $I = (I_{i,j})_{1 \leq i,j \leq l}$
the incidence matrix where $I_{i,j} = 2\delta_{i,j} - C_{i,j}$.
Let $B = (B_{i,j})_{1 \leq i,j \leq l} = DC$ ({i.e.} $B_{i,j}=r^{\vee}(\alpha_i, \alpha_i)$).
We def\/ine $l \times l$ matrices $C(q,t)$, $D(q,t)$, $B(q,t)$ and $M(q,t)$ as follows
\begin{gather}
C_{i,j}(q,t)  = \big(q^{r_i}t^{-1} + q^{-r_i}t\big)\delta_{i,j} - [I_{i,j}]_{q}, \nonumber \\
D(q,t)  = \operatorname{diag}([r_1]_q, [r_2]_q, \dots, [r_l]_q), \nonumber \\
B(q,t)  = D(q,t)C(q,t), \nonumber \\
M(q,t)  = D(q,t)C(q,t)^{-1}  = D(q,t)B(q,t)^{-1}D(q,t), \label{Mqt}
\end{gather}
where we use the standard notation $[n]_q = {q^n - q^{-n} \over q - q^{-1}}$.

Let $\mathcal{H}_{q,t}$ be the Heisenberg algebra with generators
$a_i[n]$ and $y_i[n]$ ($i = 1,2, \dots, l$; $n \in \mathbb{Z}$) with the following relations
\begin{gather}
[a_i[n], a_j[m]]  = {1 \over n}\big(q^{r_i n} - q^{-r_i n}\big)\big(t^n - t^{-n}\big)B_{i,j}\big(q^n, t^n\big)\delta_{n,-m},  \nonumber \\
[a_i[n], y_j[m]]  = {1 \over n}\big(q^{r_i n} - q^{-r_i n}\big)\big(t^n - t^{-n}\big)\delta_{i,j}\delta_{n,-m},  \nonumber \\
[y_i[n], y_j[m]]  = {1 \over n}\big(q^{r_i n} - q^{-r_i n}\big)\big(t^n - t^{-n}\big)M_{i,j}\big(q^n, t^n\big)\delta_{n,-m}. \label{y-y}
\end{gather}
Note that we have
\begin{gather*}
a_j[n] =  \sum_{i=1}^{l}C_{i,j}\big(q^n, t^n\big)y_i[n].
\end{gather*}

Introduce the generating series:
\begin{gather*}
Y_i(z) := {:}\text{exp}\biggl(\sum_{m \ne 0}y_i[m]z^{-m} \biggr){:}.
\end{gather*}
Here we have used the standard notation for the normal ordering ${:}\cdots{:}$ for the
Heisenberg generators def\/ined as follows.  We call the negative modes $a_i[-n]$, $y_i[-n]$ ($n >0$)  creation operators and
positive modes~$a_i[n]$, $y_i[n]$ ($n >0$)  annihilation operators.
Then the normal ordered product  ${:}{\mathcal O}{:}$ of an operator~$\mathcal O$ is
obtained by moving all the creation operators to the left and the annihilation operators to the right.
For example we have
 \begin{gather*}
Y_i(z) = \exp\biggl(\sum_{m >0}y_i[-m]z^{m} \biggr)
 \exp \biggl(\sum_{m >0}y_i[m]z^{-m} \biggr) .
\end{gather*}

\begin{rem}
In \cite{FR}, suitable zero mode factors are included in the generating series
$Y_i(z)$ to ensure reasonable commutation relations with the screening operators.
In this paper, however, we omit writing them since we do not need any arguments
based on the screening operators.
\end{rem}

We def\/ine a set $J$ and f\/ields $\Lambda_i(z)=\Lambda^{(X)}_i(z)$, $i \in J$, for $X=C_l$ and $D_l$.
\begin{itemize}\itemsep=0pt
\item[(i)] The $C_l$ series. $J:= \{1,2,\dots, l, \overline{l}, \overline{l-1}, \dots, \overline{1}  \}$,
\begin{gather}
\Lambda_i(z)  := {:}Y_i\big(zq^{-i+1}t^{i-1}\big)Y_{i-1}\big(zq^{-i}t^i\big)^{-1}{:}, \qquad i=1, 2, \dots, l,  \nonumber \\
\Lambda_{\overline{i}}(z)  := {:}Y_{i-1}\big(zq^{-2l+i-2}t^{2l-i}\big)Y_{i}\big(zq^{-2l+i-3}t^{2l-i+1}\big)^{-1}{:}, \qquad i=1, 2, \dots, l. \label{Lam-C}
\end{gather}
\item[(ii)] The $D_l$ series. $J:= \{1,2,\dots, l, \overline{l}, \overline{l-1}, \dots, \overline{1}  \}$,
\begin{gather}
\Lambda_i(z)  := {:}Y_i\big(zq^{-i+1}t^{i-1}\big)Y_{i-1}\big(zq^{-i}t^i\big)^{-1}{:}, \qquad i=1, 2, \dots, l-2, \nonumber \\
\Lambda_{l-1}(z) := {:}Y_l\big(zq^{-l+2}t^{l-2}\big)Y_{l-1}\big(zq^{-l+2}t^{l-2}\big)Y_{l-2}\big(zq^{-l+1}t^{l-1}\big){:}, \nonumber \\
\Lambda_{l}(z)  := {:}Y_l\big(zq^{-l+2}t^{l-2}\big)Y_{l-1}\big(zq^{-l}t^{l}\big)^{-1}{:}, \nonumber \\
\Lambda_{\overline{l}}(z)  :=  {:}Y_{l-1}\big(zq^{-l+2}t^{l-2}\big)Y_{l}\big(zq^{-l}t^{l}\big)^{-1}{:}, \nonumber \\
\Lambda_{\overline{l-1}}(z) := {:}Y_{l-2}\big(zq^{-l+1}t^{l-1}\big)Y_{l-1}\big(zq^{-l}t^{l}\big)^{-1}Y_{l}\big(zq^{-l}t^{l}\big){:}, \nonumber \\
\Lambda_{\overline{i}}(z)  := {:}Y_{i-1}\big(zq^{-2l+i+2}t^{2l-i-2}\big)Y_{i}\big(zq^{-2l+i+1}t^{2l-i-1}\big)^{-1}{:}, \qquad i=1, 2, \dots, l-2. \!\!\!\label{Lam-D}
\end{gather}
\end{itemize}
\begin{dfn}[\cite{FR}]
Def\/ine the f\/irst generating f\/ields $T^{(X)}(x,z)$
of the deformed~$\mathcal{W}$
algebras of type $X=C_l$ or $D_l$ with the independent indeterminates $x = (x_1, x_2, \dots, x_l)$:
\begin{gather}
T^{(X)}(x,z):=x_1\Lambda_1(z)+\cdots+x_l\Lambda_l(z)+\frac{1}
{x_l}\Lambda_{\bar{l}}(z)+\cdots+
\frac{1}{x_1}\Lambda_{\bar{1}}(z). \label{TX}
\end{gather}
\end{dfn}

\begin{rem}
It is not an easy task to def\/ine the deformed $\mathcal{W}$ algebras purely
in terms of the generators and relations,
except for some simple cases  such as the deformed Virasoro algebra.
One of the simplest bypass ways is to regard the deformed~$\mathcal{W}$ as the algebra generated by
the~$T^{(X)}(x,z)$ given in terms of the Heisenberg generators.
\end{rem}

\begin{rem}
The $x_i$'s $(i = 1,2, \dots, l)$  correspond to
the zero mode factors, and they paramet\-ri\-ze the highest weight condition
for the representation of the~$\mathcal{W}$ algebras.
\end{rem}

\begin{lem}\label{OPE} For $C_l$ and $D_l$ series
we obtain the following operator product expansions
\begin{gather*}
 f(w/z)\Lambda_i(z)\Lambda_j(w)=\gamma_{i,j}(z,w)\, {:}\Lambda_i(z)\Lambda_j(w){:},
\end{gather*}
where
\begin{gather*}
f(z) = f^{(X)}(z) := \exp \left(-\sum_{n=1}^{\infty}(q^n - q^{-n})(t^n - t^{-n})M_{1,1}(q^n, t^n)z^n \right),
\end{gather*}
and $\gamma_{i,j}(z,w) = \gamma_{i,j}^{(X)}(z,w) $ for $X = C_l $ or $ D_l$ given by
\begin{gather*}
\gamma_{i,j}^{(C_l)}(z,w)
=\begin{cases}
1,& i=j, \\
\gamma(w/z),& i \prec j,\ j\neq \bar{i}, \\
\gamma(z/w),& i \succ j, \ \bar{j}\neq i, \\
\gamma(w/z)\gamma\big(q^{2i-2l-2}t^{-2i+2l}w/z\big),& i \prec j,\ j=\bar{i}, \\
\gamma(z/w)\gamma\big(q^{-2i+2l+2}t^{2i-2l}z/w\big),&  i \succ j, \ \bar{j}=i, \\
\end{cases} \\
\gamma_{i,j}^{(D_l)}(z,w)
=\begin{cases}
1& i=j, \\
\gamma(w/z),& i \prec j, \ j\neq \bar{i}, \\
\gamma(z/w),& i \succ j, \ \bar{j}\neq i, \\
\gamma(w/z)\gamma\big(q^{2i-2l+2}t^{-2i+2l-2}w/z\big),&  i \prec j, \ j=\bar{i}, \\
\gamma(z/w)\gamma\big(q^{-2i+2l-2}t^{2i-2l+2}z/w\big),& i \succ j, \ \bar{j}=i . \\
\end{cases}
\end{gather*}
Here we have use the notation
$\gamma(z)=\frac{(1-t^2z)(1-z/q^2)}
{(1-z)(1-zt^2/q^2)}$.
\end{lem}

A proof of Lemma \ref{OPE} can be obtained by straightforward but pretty lengthy calculations using
(\ref{Mqt}), (\ref{y-y}), (\ref{Lam-C}), (\ref{Lam-D}).  Therefore we safely can omit the detail.

Let $| 0 \rangle$ be the vacuum vector satisfying
the annihilation conditions $a_i[n] | 0 \rangle = 0$ for all~$i $ and $n > 0$.
Let $\mathcal{F}$ be the Fock module obtained by
inducing up the one dimensional representation~${\mathbb C}| 0 \rangle$ of the algebra of
annihilation operators to the whole Heisenberg algebra.
Then one can check that the Fourier modes of the generator $T^{(X)}(x,z)$ acting on $\mathcal{F}$ are well def\/ined.
Let~$\langle 0 |$ to be the dual vacuum satisfying $\langle 0 |a_i[-n] = 0$ for all $i $ and $n > 0$.

Now we consider the correlation functions
$\langle 0|T(x,z_1)T(x,z_2)\cdots T(x,z_r)|0\rangle$
of types~$C_l$ and~$D_l$
with the normalization
$\langle 0 |\Lambda_i(z)| 0 \rangle = 1$.

\begin{prp}\label{CorrGamma}
Let $X=C_l$ or $D_l$.
Set $I_l := \{ 1,2,\dots,l,\overline{l}, \overline{l-1},\dots,\overline{1} \}$.
Let $x_1,\dots,x_l$ be indeterminates and set
$x_{\overline{i}}=1/x_i$, $1\leq i\leq l$.
We have
\begin{gather*}
\prod_{i<j}f^{(X)}(z_j/z_i) \cdot
\langle 0|T^{(X)}(x,z_1)T^{(X)}(x,z_2)\cdots T^{(X)}(x,z_r)|0 \rangle\\
\qquad{} =F^{(X)}(x_1,\dots,x_l|z_i,\dots,z_r|q,t),
\end{gather*}
where
\begin{gather*}
F^{(X)}(x_1,\dots,x_l|z_i,\dots,z_r|q,t)
 =
\sum_{\varepsilon_1, \varepsilon_2, \dots, \varepsilon_r \in
I_l} x_{\varepsilon_1}x_{\varepsilon_2} \cdots x_{\varepsilon_r}
\prod_{1 \leq i < j \leq r}
\gamma_{\varepsilon_i, \varepsilon_j}(z_{i}, z_{j}).
\end{gather*}
\end{prp}

\begin{proof}[Proof of Proposition \ref{CorrGamma}]
In view of the expression~(\ref{TX})
we only need to know the matrix elements
\begin{gather}
\langle 0|\Lambda^{(X)}_{\varepsilon_1}(z_1)\Lambda^{(X)}_{\varepsilon_2}(z_2)\cdots \Lambda^{(X)}_{\varepsilon_r}(z_r)|0 \rangle,
\label{LamLam}
\end{gather}
for any f\/ixed $\varepsilon_1,\ldots,\varepsilon_r\in I_l$.
Then  Proposition~\ref{OPE} and the normal ordering rule imply that
\begin{gather*}
(\ref{LamLam})=\prod_{i<j} \big(f^{(X)}(z_j/z_i) \big)^{-1}\cdot
\prod_{1 \leq i < j \leq r}
\gamma_{\varepsilon_i, \varepsilon_j}(z_{i}, z_{j}).\tag*{\qed}
\end{gather*}
\renewcommand{\qed}{}
\end{proof}

\begin{rem}\label{rem:b3}
It is clearly seen from the def\/inition that $F^{(X)}(x_1,\dots,x_l|z_i,\dots,z_r|q,t)$ is
a~symmetric rational function in~$z_i$'s.
We conjecture that  the $F^{(X)}(x_1,\dots,x_l|z_i,\dots,z_r|q,t)$ is
a~symmetric Laurent polynomial in $x_i$'s associated with the
Weyl group of corresponding type~$X$ ($C_l$~or~$D_l$). For the
case of type~$A_l$, we better understand the situation
due to the theory of the shuf\/f\/le algebra (we refer the reader to~\cite{FHSSY}).
\end{rem}

\begin{dfn}\label{DefOfCorrFunc}
By principally specializing the $z_i$'s,
set
\begin{gather*}
\Phi^{(X)}_r(x|q,t)=F^{(X)}\big(x_1,\dots,x_l|q^{r-1},q^{r-2},\dots,1|q^{1/2},q^{1/2}t^{-1/2}\big).
\end{gather*}
\end{dfn}

\begin{rem} Remark \ref{rem:b3} implies that $\Phi^{(X)}_r(x|q,t)$ is
a symmetric Laurent polynomial in $x_i$'s associated with the
Weyl group of corresponding type~$X$ ($C_l$~or~$D_l$).
It is easy to check that the terms which do not vanish under the principal
specialization correspond exactly to the set of semi-standard tableaux
of type~$X$.
\end{rem}

\begin{thm}\label{soukan1}
We have
\begin{itemize}\itemsep=0pt
\item[$(i)$]
$\Phi^{(C_l)}_{r}(x|q,t) =  P_{(r)}^{(C_l)}(x;q,t,t^2/q)$,
\item[$(ii)$]
$\Phi^{(D_l)}_{r}(x|q,t) =  P_{(r)}^{(D_l)}(x;q,t)$.
\end{itemize}
\end{thm}

\begin{proof} Straightforward calculation of
$\Phi^{(X)}_r(x|q,t)$
gives us (\ref{MacOfCt^2/q}) for $X=C_l$ and~(\ref{MacOfD}) for $X=D_l$.
\end{proof}

\appendix
\section[Macdonald polynomials of types $C_n$ and $D_n$]{Macdonald polynomials of types $\boldsymbol{C_n}$ and $\boldsymbol{D_n}$}\label{appendixA}

We recall brief\/ly the Koornwinder polynomials, and the def\/initions
of the Macdonald polynomials of types~$C_n$ and~$D_n$
as degenerations of the Koornwinder polynomials.

\subsection{Koornwinder polynomials} \label{appendixA.1}

Let $(a,b,c,d,q,t)$ be a set of complex parameters with $|q|<1$.
Set  $\alpha=(abcdq^{-1})^{1/2}$ for simplicity.
Let $x=(x_1,\ldots,x_n)$ be a set of independent indeterminates.
Koornwinder's $q$-dif\/ference operator  ${\mathcal D}_x={\mathcal D}_x(a,b,c,d|q,t)$ is def\/ined by~\cite{Ko}
\begin{gather*}
{\mathcal D}_x=
\sum_{i=1}^n {(1-ax_i)(1-bx_i)(1-cx_i)(1-dx_i)\over
\alpha t^{n-1}(1-x_i^2)(1-qx_i^2)}
\prod_{j\neq i} {(1-t x_ix_j)(1-t x_i/x_j)\over (1-x_ix_j)(1-x_i/x_j)}
(T_{q,x_i}-1) \\
{}+
\sum_{i=1}^n {(1-a/x_i)(1-b/x_i)(1-c/x_i)(1-d/x_i)\over
\alpha t^{n-1}(1-1/x_i^2)(1-q/x_i^2)}
\prod_{j\neq i} {(1-t x_j/x_i)(1-t /x_ix_j)\over (1-x_j/x_i)(1-1/x_ix_j)}
(T_{q^{-1},x_i}-1),
\end{gather*}
where  \looseness=-1 $T_{q^{\pm 1},x_i}f(x_1,\ldots,x_i,\ldots ,x_n)=f(x_1,\ldots,q^{\pm 1}x_i,\ldots ,x_n)$.
The Koornwinder polynomial $P_\lambda(x)=P_\lambda(x;a,b,c,d,q,t)$
with partition $\lambda=(\lambda_1,\ldots,\lambda_n)$
(i.e., $\lambda_i\in \mathbb{Z}_{\geq 0},\lambda_1\geq \cdots\geq \lambda_n$)
is uniquely characterized by the two conditions
(a)~$P_\lambda(x)$ is a $S_n \ltimes (\mathbb{Z}/2\mathbb{Z})^n$ invariant Laurent
polynomial having the triangular expansion in terms of the monomial
basis $(m_\lambda(x))$ as $P_\lambda(x)=m_\lambda(x)+\mbox{lower terms}$,
(b)~$P_\lambda(x)$ satisf\/ies ${\mathcal D}_x P_\lambda(x)=d_\lambda P_\lambda(x)$.
The eigenvalue is given by
\begin{gather*}
d_\lambda=\sum_{j=1}^n \langle abcdq^{-1}t^{2n-2j}q^{\lambda_j}\rangle
\langle q^{\lambda_j}\rangle
=
\sum_{j=1}^n \langle \alpha t^{n-j}q^{\lambda_j}; \alpha t^{n-j}
\rangle,
\end{gather*}
where we used the notation
$\langle x\rangle=x^{1/2}-x^{-1/2}$ and
$\langle x;y\rangle=\langle xy\rangle\langle x/y\rangle=x+x^{-1}-y-y^{-1}$
for simplicity of display.

\subsection[Macdonald polynomials of types $C_n$ and $D_n$]{Macdonald polynomials of types $\boldsymbol{C_n}$ and $\boldsymbol{D_n}$}
\label{appendixA.2}
We consider some degeneration of the Koornwinder polynomials to
the Macdonald polynomials. As for the details, we refer the readers to~\cite{Ko} and~\cite{Mac1}.
Specializing the parameters in the Koornwinder polynomial $P_\lambda(x;a,b,c,d,q,t)$ as
\begin{gather*}
(a,b,c,d,q,t)\rightarrow \big({-}b^{1/2},ab^{1/2}, -q^{1/2}b^{1/2},q^{1/2}ab^{1/2},q,t\big),
\end{gather*}
we obtain the Macdonald polynomial of type $(BC_n,C_n)$ \cite{Ko}
\begin{gather*}
P^{(BC_n,C_n)}_\lambda(x;a,b,q,t)=P_\lambda\big(x;-b^{1/2},ab^{1/2}, -q^{1/2}b^{1/2},q^{1/2}ab^{1/2},q,t\big).
\end{gather*}
Namely, setting
\begin{gather*}
D^{(BC_n,C_n)}_x=\sum_{\sigma_1,\dots,\sigma_n=\pm 1}
\prod_{i=1}^n
{(1-a b^{1/2} x_i^{\sigma_i})(1+b^{1/2} x_i^{\sigma_i})\over 1-x_i^{2\sigma_i}}\cdot
\prod_{1\leq i<j\leq n}
{1-t x_i^{\sigma_i} x_i^{\sigma_j}\over 1-x_i^{\sigma_i} x_i^{\sigma_j}} \cdot
\prod_{i=1}^n
T_{q^{\sigma_i/2},x_i},
\end{gather*}
we have{\samepage
\begin{gather*}
P^{(BC_n,C_n)}_\lambda(x)=m_\lambda(x)+\mbox{lower terms},\\
D^{(BC_n,C_n)}_x P^{(BC_n,C_n)}_\lambda(x)=
(ab)^{n/2}t^{n(n-1)/4}\left( \sum_{\sigma_1,\dots,\sigma_n=\pm 1}
 s_1^{\sigma_1/2}\cdots s_n^{\sigma_n/2}\right)
P^{(BC_n,C_n)}_\lambda(x),
\end{gather*}
where $s_i=ab t^{n-i}q^{\lambda_i}$.}

Macdonald polynomials of type $C_n$ and type $D_n$ are obtained as follows
\begin{gather*}
P^{(C_n)}_\lambda(x;b,q,t)=P^{(BC_n,C_n)}_\lambda(x;1,b,q,t),\\
P^{(D_n)}_\lambda(x;q,t)=P^{(BC_n,C_n)}_\lambda(x;1,1,q,t).
\end{gather*}

\subsection*{Acknowledgements}

This study was carried out within The National Research University Higher
School of Economics Academic Fund Program in 2013--2014, research
grant No.12-01-0016k.
Research of J.S.\ is supported by the Grant-in-Aid for Scientif\/ic
Research C-24540206.
The f\/inancial support from the Government of the Russian Federation within
the framework of the implementation of the 5-100 Programme Roadmap of the
National Research University  Higher School of Economics is acknowledged.
The authors sincerely thank the anonymous referees for valuable comments and informing them of the connection between
one of our results (Theorem~\ref{koushiki-2}) and the known fact obtained in~\cite{LSW}.

\pdfbookmark[1]{References}{ref}
\LastPageEnding

\end{document}